\documentclass[12pt]{article}
\usepackage[utf8]{inputenc}
\usepackage{amsmath}
\usepackage{amsfonts}
\usepackage{amssymb}
\usepackage{theorem,pifont}
\usepackage{graphicx}
\usepackage{color}
\usepackage{pgf,tikz}
\usetikzlibrary{arrows}
\usepackage[nodraft]{symbols}

\title{\bf Global 
Stabilization of Chemostats 
\\ 
with Nonzero Mortality and Substrate Dynamics
}
\author{Iasson Karafyllis, Epiphane Loko\footnote{Universite Paris-Saclay, CNRS, CentraleSupélec, Laboratoire des signaux et systemes, 91190, Gif-sur-Yvette, France. \texttt{dagbegnon-epiphane.loko@centralesupelec.fr}}, Antoine Chaillet and Miroslav Krsti{\'c}}
\date{\today}

\definecolor{darkpastelgreen}{rgb}{0.01, 0.75, 0.24}

\begin{document}
\title{Global Stabilization of Chemostats with\\ Nonzero Mortality and Substrate Dynamics}
\author{Iasson Karafyllis,  Epiphane Loko\footnote{Corresponding author: Universite Paris-Saclay, CentraleSupélec, Laboratoire des signaux et systemes, 91190, Gif-sur-Yvette, France. \texttt{dagbegnon-epiphane.loko@centralesupelec.fr}}, Miroslav Krsti\'c, and Antoine Chaillet  
}

\maketitle

\textbf{Abstract:} In ``chemostat''-type population models that incorporate substrate (nutrient) dynamics, the dependence of the birth (or growth) rate on the substrate concentration introduces nonlinear coupling that creates a challenge for stabilization that is global, namely, for all positive concentrations of the biomass and nutrients. This challenge for global stabilization has been overcome in the literature using relatively simple feedback when natural mortality of the biomass is absent. However, under natural mortality, it takes fortified, more complex feedback, outside of the existing nonlinear control design toolbox, to avoid biomass extinction from nutrient-depleted initial conditions. Such fortified feedback, the associated control Laypunov function design, and Lyapunov analysis of global stability are provided in this paper. We achieve global stabilization for two different chemostat models: (i) a lumped model, with two state variables, and (ii) a three-state model derived from an age-structured infinite-dimensional model. The proposed feedback stabilizers are explicit, applicable to both the lumped and the age-structured models, and coincide with simple feedback laws proposed in the literature when the mortality rate is zero. Global stabilization means subject to constraints:  all positive biomass and nutrient concentrations are within the region of attraction of the desired equilibrium, and, additionally, this is achieved with a dilution input that is guaranteed to remain positive. For the lumped case with Haldane kinetics, we show that the reproduction rate dominating the mortality (excluding the reproduction and mortality being in balance) is not only sufficient but also necessary for global stabilization. The obtained results are illustrated with simple examples.

%


\section{Introduction}
\paragraph{Motivation.}
	The chemostat is a well-known model of continuous microbial bioreactors (see \cite{Harmand2017-chemostat}, \cite{Smith1995-Theory}, \cite{Dochain2013-Automatic}, \cite{Bastin2013-Line}). Chemostat models have been used in several environmental technology and ecological processes, such as the wastewater treatment to degrade pollutant and the production of biomass \cite{Veldkamp1977-ecological}. The chemostat has been studied both as a dynamical system and as a control system. In the framework of the chemostat as a dynamical system, several studies exist in the literature.  In \cite{Mazenc2024-stability} for instance the authors consider one-species delayed chemostat model with a periodic input of a single substrate and they provide condition under which the global asymptotic stability of the non-trivial periodic solution can be achieved (with small delay) (see also \cite{Amster2020-dynamics}, \cite{Ye2022-periodic}, \cite{Ye2023-dynamical} for the existence and uniqueness of periodic solution of this model). In \cite{Butler1985-Mathematical}, \cite{Wolkowicz1992-Global} and \cite{Mazenc2012-stability}, the authors consider a chemostat model with multiple competitive species and they provide conditions on the specific growth rate function to ensure the global asymptotic stability (see also \cite{Wang2006-Delayed} for the same model with delay). A coexistence constraint on the input of the same multiple species chemostat model 
    is presented in \cite{Robledo2012-global} where the authors identify the globally asymptotically stable equilibrium.  
    In \cite{RAHAHA}, the impact of the presence of a buffer tank on the concentration equilibria and their stability is assessed.
    In \cite{Beauthier2015-Input} the authors solve the problem of coexistence of species in competition in chemostat by using an LQ-optimal control. A single-species  age-structured chemostat model is studied in \cite{Toth2006-Limit} where the authors provide conditions for the existence of limit cycles. A stability result is also investigated in  \cite{Borisov2022-stability} for a chemostat model which describes the biodegradation of the binary mixture of phenol and sodium salicylate. It is also shown in \cite{Harmand2020-Increasing} that increasing dilution rate can globally stabilize chemostat model like two-step mass-balance
    biological process (see \cite{Harmand2006-Dynamical} for a similar model presented as autocatalytic system).  
    
\paragraph{Literature on chemostat control.}  The chemostat as a control system has attracted the attention of many researchers in the literature. In \cite{Karafyllis2009-Relaxed}, \cite{Karafyllis2008-Vector} for instance, the robust global feedback stabilzation for a chemostat model with uncertainty is studied (see also \cite{Mailleret2001-A-simple} for the same model without uncertainty). In the context of global stabilization of the chemostat model with delays, a positive-valued feedback is constructed in \cite{Karafyllis2012-New} based on a new small gain method and also in \cite{Mazenc2010-Stabilization}, \cite{Mazenc2010-Stabilization-2-species} using Lyapunov-Krasovskii arguments.  In \cite{Gouze2006-feedback} the authors extend the feedback controller proposed in \cite{De2003-Feedback} to a chemostat model with two species (see also \cite{Dimitrova2012-Nonlinear} for two species chemostat model). In \cite{Borisov2020-Global}, a piecewise constant feedback is constructed to ensure the asymptotic stabilization of a chemostat model which describes the methane fermentation of organic wastes in a continuously stirred tank bioreactor. A particular class of chemostat models which is investigated in the literature is the age-structure chemostat which describes the evolution of a certain population by taking into account both the age and the time (see \cite{Brauer2012-Mathematical}, \cite{Boucekkine2011-Optimal},\cite{Rundnicki1994-Asymptotic}). There are several studies on feedback stabilization problems for age-structured chemostat models such as  \cite{Karafyllis2017-Stability}, \cite{Schmidt2018-Yield},  \cite{Haacker2024-stabilization}, \cite{Veil2024-Stabilization} (which make use of Control Lyapunov functionals) and also extensive studies in optimal control: see \cite{Boucekkine2011-Optimal}, \cite{Feichtinger2003-Optimality}, \cite{Sun2014-Optimal}.
    
\paragraph{Contribution.}    
In chemostat models that incorporate substrate (nutrient) dynamics, the dependence of the birth (or growth) rate on the substrate concentration introduces multiplicative influence of the substrate in the biomass dynamics, which creates a challenge for {global stabilization}, namely {valid} for all positive concentrations of the biomass and nutrients. This challenge for global stabilization has been overcome in the literature using relatively simple feedback when natural mortality of the biomass is absent. However, under natural mortality, it takes fortified, more complex feedback, outside of the existing nonlinear control design toolbox, to avoid biomass extinction from nutrient-depleted initial conditions. Such fortified feedback, the associated design of a  strict control Laypunov function, and Lyapunov analysis of global stability are provided in this paper. Importantly, the fortification of feedback is with a nonlinear term whose (positive) gain can be taken arbitrarily small, hence, the ``fortification'' to obtain globality is subtle, rather than brute force. 

We achieve global stabilization for two different chemostat models: (i) lumped model, with two state variables, and (ii) a three-state model which is derived by an age-structured model, infinite-dimensional model. What is most interesting is that, in addition to being given by explicit formulae, the proposed feedback laws (a) are the same for both the lumped and the age-structured case, and (b) coincide with the simple feedback law proposed in \cite{Mailleret2001-A-simple,Karafyllis2008-Vector} when the mortality rate is zero. 

    From a control-theoretic point of view, the control study that is performed in the present work is challenging because: 1) the equilibrium points that are stabilized can be open-loop unstable, 2) all state and input constraints (positivity of the state and input) are taken into account. Our feedback design methodology provides and uses Control Lyapunov Functions (CLFs) for the chemostat models which can also be used for further robustness analysis. The Lyapunov function design methodology is combined with appropriate changes of coordinates that allow us to study the feedback stabilization problem in standard linear spaces. 

Our study is not only of interest in the control theoretic sense but also in the context of mathematical biology.  We show, in particular, for the lumped case with Haldane kinetics for the specific growth rate (but also for any other case where the kinetic equation for the specific growth rate is given by a concave function), that the  condition
that 
\begin{center}
    ``the specific growth rate calculated at the inlet value of the substrate  concentration is greater than the mortality rate"
    \end{center}
\noindent is not only sufficient for the existence of a globally stabilizing feedback, but, when strictly violated, makes global stabilization impossible.
    
    The above sufficient condition seems to have physical meaning and to be reflecting the biological laws of the chemostat model. The study of a nonlinear control system for which one can give necessary and sufficient conditions for the existence of global feedback stabilizers is rare in the literature. For the age-structured case, we also show that the proposed sufficient condition for the existence of a global feedback stabilizer becomes identical to the above condition for certain limiting cases. 

\paragraph{Organization.} Section \ref{Sec-Simplest} is devoted to the simple lumped chemostat model with non-zero mortality rate. Section \ref{Sec-Age} presents the age-structured model and the feedback stabilization result. In Section \ref{Sec-Simulation} we illustrate our main results by means of simple examples and in Section \ref{Sec-Proofs}, we provide the proofs of our main results. Finally, in Section \ref{Conclusions} we give the concluding remarks of the present work. 

\vspace{3mm}
\emph{Acronyms and notations.} 
$\mathbb{R}$ stands for the set of all real numbers, and for any $n\in \mathbb{N}\setminus\{0\},$ $\mathbb{R}^n$ denotes the $n-$dimensional real space. $\mathbb{R}_+$ stands for the non-negative real numbers. $C^1(\mathbb{R}_+)$ is the set of continuously differentiable function $f:\mathbb{R_+}\rightarrow\mathbb{R}_+.$ Given any continuously differentiable function $V:\mathbb{R}^n\rightarrow\mathbb{R},$ $\dot{V}(x)$ denotes its derivative in times which is $\dot{V}(x):=\nabla V(x) \dot{x}$ where $\nabla V(x):=\left(\frac{\partial V}{\partial x_1}(x),\cdots, \frac{\partial V}{\partial x_n}(x)\right)$ is the gradient of $V$. PDE: Partial Differential Equation, ODE: Ordinary Differential Equation.

\section{A Simple Chemostat Model with Mortality}\label{Sec-Simplest}

Consider the chemostat model:
    \begin{align}
    \begin{cases}
        \dot{X}=(p_0\mu(S)-b-D)X\\
        \dot{S}=D(S_{\text{in}}-S)-\mu(S)X,
    \end{cases}\label{system-simple1}
    \end{align}
\noindent where $X(t)>0$ is the concentration of the microbial population in the chemostat at time $t,$ $S(t)>0$ is the limiting substrate concentration, $S_{\text{in}}>0$ is the inlet concentration of the substrate, $D(t)>0$ is the dilution rate, $p_0>0$ is a constant (a yield coefficient), $\mu$ is the specific growth rate function and $b\geq 0$ is the mortality rate. The function $\mu: \mathbb{R}_+\rightarrow \mathbb{R}_+$ is a bounded, $C^1(\mathbb{R}_+)$ function with $\mu(0)=0$ and $\mu(S)>0$ for $S>0.$ 

We consider model \eqref{system-simple1} as a control system with input $D>0$ and state space the open set $(0,+\infty)\times (0,S_{\text{in}}),$ i.e, we consider \eqref{system-simple1} with
\begin{align*}
    (X,S)\in (0,+\infty)\times (0,S_{\text{in}}).
\end{align*}
We suppose that, for some constant $D^*>0$, there exists an equilibrium point when $D(t)\equiv D^*$, i.e a point $(X^*,S^*)\in (0,+\infty)\times (0,S_{\text{in}})$ satisfying
\begin{align}
    &p_0\mu(S^*)=b+D^*,\quad X^*=\frac{D^*(S_{\text{in}}-S^*)}{\mu(S^*)}.\label{eqpt-mu1}
\end{align}
Our objective in this section is to stabilize globally this equilibrium point by means of a locally Lipschitz feedback law. In order to do that, we first use the following diffeomorphism
\begin{align}
    x_1=\ln\left(\frac{\mu(S^*)X}{D^*(S_{\text{in}}-S)}\right),\quad 
    x_2=\ln\left(\frac{S(S_{\text{in}}-S^*)}{S^*(S_{\text{in}}-S)}\right)\label{x_1-x_2}
\end{align} with inverse
   \begin{align}
    X=\frac{D^*S_{\text{in}}(S_{\text{in}}-S^*)e^{x_1}}{\mu(S^*)\left(S_{\text{in}}-S^*+S^*e^{x_2}\right)},\quad
    S=\frac{S^*S_{\text{in}}e^{x_2}}{S_{\text{in}}-S^*+S^*e^{x_2}} \label{X-S-1}.
\end{align} 
This diffeormorphism maps $(0,+\infty)\times (0,S_{\text{in}})$ onto $\mathbb{R}^2$ and the equilibrium point $(X^*,S^*)\in (0,+\infty)\times (0,S_{\text{in}})$ to $0\in \mathbb{R}^2.$ Therefore we transform the chemostat model \eqref{system-simple1} to the following control system
\begin{subequations}
    \begin{align}
    &\dot{x}_1=b(g(x_2)-1)+D^*g(x_2)(1-e^{x_1})\label{x_1-simple}\\
    &\dot{x}_2=p(x_2)(D-D^*g(x_2)e^{x_1})\label{x_2-simple}\\
    &x=(x_1,x_2)^{\top}\in \mathbb{R}^2, D>0,
\end{align}\label{systeme-simple2}
\end{subequations} where for all $x_2\in \mathbb{R}$
\begin{subequations}
    \begin{align}
    g(x_2)&:=\frac{1}{\mu(S^*)}\mu\left(\frac{S_{\text{in}}}{p(x_2)}\right)\label{g}\\
    p(x_2)&:=\kappa e^{-x_2}+1\label{p}\\
    \kappa&:=\frac{S_{\text{in}}-S^*}{S^*}>0.\label{kappa}
\end{align}\label{g-p-kappa}
\end{subequations}
The stability of the equilibrium can easily be characterized, as stated next.

\begin{fact}\label{fact1}
The equilibrium point $(X^*,S^*)\in (0,+\infty)\times (0,S_{\text{in}})$ of \eqref{systeme-simple2} is unstable when $\mu'(S^*)<0$.
\end{fact}

\begin{proof}
The linearization of \eqref{systeme-simple2} at $0\in \mathbb{R}^2$ is given by
\begin{align}
    \begin{pmatrix}
        \dot{x}_1\\
        \dot{x}_2
    \end{pmatrix}=\begin{pmatrix}
        &-D^*g(0)\quad &bg'(0)\\
        &-D^*p(0)\quad &-D^*p(0)g'(0)
    \end{pmatrix}\begin{pmatrix}
        x_1\\
        x_2
    \end{pmatrix}.\label{linear1}
\end{align} 
The characteristic polynomial of the matrix involved in \eqref{linear1} reads
\begin{align*}
    f(s)=s^2+D^*(g(0)+p(0)g'(0))s+D^*p(0)g'(0)(b+D^*g(0)).
\end{align*}
 Thus, we can see that the linearization \eqref{linear1} of \eqref{systeme-simple2} is unstable when $g'(0)<0.$ Equivalently, due to definitions \eqref{g}, \eqref{p}, the equilibrium point $(X^*,S^*)\in (0,+\infty)\times (0,S_{\text{in}})$ is unstable when $\mu'(S^*)<0.$
 \end{proof}

We solve next the feedback stabilization problem for the chemostat model \eqref{system-simple1} by studying the equivalent system \eqref{systeme-simple2} under the following assumption.\\ \\
\textbf{(A)} For all $S\in [S^*,S_{\text{in}}],$ the following inequality holds
\begin{align}
    p_0\mu(S)>b.\label{Assumption-A}
\end{align}
Assumption \textbf{(A)} was already used in \cite{Karafyllis2012-New} for a time-delay chemostat model and imposes a limitation on the mortality rate $b$.  We discuss the need of imposing assumption \textbf{(A)} below. Under assumption \textbf{(A)} we have the following result, whose proof is provided in Section \ref{Proof-Th1}. 

\begin{theorem}\label{Th1}
    Consider system \eqref{systeme-simple2} under assumption \textbf{(A)}. Then for every $\delta>0$ and $\alpha \in [0,1)$ the feedback law defined as
    \begin{align}
        D(x)=&(\kappa+1)D^*g(x_2)\frac{e^{x_1-x_2}}{p(x_2)}+ \delta b \begin{cases}
            |g(x_2)-1|^{1+\alpha}, \quad \text{if}\quad x_2\leq 0\\
            0, \quad \text{if}\quad x_2>0
        \end{cases}\label{Feedback1}
    \end{align} 
    is locally Lipschitz and achieves global asymptotic stabilization of $0\in \mathbb{R}^2$. Moreover, if $\alpha >0$ then the feedback law \eqref{Feedback1} also achieves local exponential stabilization of $0\in \mathbb{R}^2$.  
\end{theorem}

The feedback law \eqref{Feedback1} in the original coordinates is given by the equation:
\begin{align}
    &D(X,S)=\frac{D^*\mu(S)X}{\mu(S^*)X^*}+\frac{\delta b}{(\mu(S^*))^{1+\alpha}}\begin{cases}
        |\mu(S)-\mu(S^*)|^{1+\alpha}, \quad \text{if}\quad S\leq S^*\\
        0, \quad \text{if} \quad S>S^*.
    \end{cases}\label{Feedback-origin}
\end{align}

Formula \eqref{Feedback-origin} shows that the feedback law is composed by two terms. The first term
\begin{align*}
    \frac{D^*\mu(S)X}{\mu(S^*)X^*}
\end{align*} coincides with the feedback proposed in \cite{Mailleret2001-A-simple,Karafyllis2008-Vector} when $b=0$ (notice that $\mu(S^*)=D^*$ when $b=0$), namely when the mortality rate is neglected.

The second term
\begin{equation}\label{eq-1}
    \frac{\delta b}{(\mu(S^*))^{1+\alpha}}\begin{cases}
        |\mu(S)-\mu(S^*)|^{1+\alpha}, \quad \text{if}\quad S\leq S^*\\
        0, \quad \text{if} \quad S>S^*
    \end{cases}
\end{equation}
     is a non-negative bounded term that vanishes when $b=0.$ Moreover, since $\delta>0$ is arbitrarily small, we can select it appropriately in order to make it as small as desired. 
     On the other hand, a higher value of $\delta$ can accelerate the convergence to the equilibrium point (see Section \ref{Sec-Simulation} below).
    
Theorem \ref{Th1} shows that the feedback law $D(x)=\frac{\mu(S)X}{X^*}$ proposed in \cite{Mailleret2001-A-simple,Karafyllis2008-Vector} for the case when $b=0$ can be complemented by an additive term in order to achieve global asymptotic stabilization, at least for the case when the mortality rate satisfies 
\begin{align*}
        0<b< p_0\mu(S),\quad \forall S\in [S^*,S_{\text{in}}].
    \end{align*}
This additional term \eqref{eq-1}
  is physically meaningful: it increases the value of the dilution rate so that the system is not left to spend much time in conditions where there is lack of sufficient nutrient (i.e, conditions where $S$ is significantly less than $S^*$).

    The following result, proved in Section \ref{Proof-Th2}, shows that assumption \textbf{(A)} is close to being necessary for the global stabilization of system \eqref{systeme-simple2}. Indeed, it shows a case where assumption \textbf{(A)} is violated and global asymptotic stabilization is not possible by means of a non-negative, locally Lipschitz feedback law.
    
    \begin{theorem}\label{Th2}
        Suppose that there exists a constant $\overline{S}\in (S^*,S_{\text{in}})$ such that
        \begin{align}
           p_0\mu(S_{\text{in}})&<p_0\mu(\overline{S})<b \label{AssumpTh2}\\
           \mu'(S)&\leq 0, \quad \forall S\in [\overline{S}, S_{\text{in}}].\label{AssumpTh2*}
        \end{align} 
        Then, there is no locally Lipschitz feedback law $D(x)\geq 0$  that achieves global stabilization of the origin of \eqref{systeme-simple2}. More precisely, for any locally Lipschitz feedback $D(x)\geq 0$, the closed-loop system admits unbounded solutions. 
    \end{theorem}
    
    Consider the chemostat model \eqref{system-simple1} with \eqref{eqpt-mu1} in the particular case when the specific growth rate $\mu$ is given by the following Haldane kinetic equation:
    \begin{align}
        \mu(S):=\frac{M S}{K+S+aS^2}, \quad \forall S \geq 0
    \end{align} where $M, K,a>0$ are constants. This case is important because it is studied extensively in the literature (see for e.g \cite{Butler1985-Mathematical}, \cite{Mazenc2010-Stabilization}, \cite{Mazenc2010-Stabilization-2-species} and the references therein). If $S_{\text{in}}$ is such that $\mu(S_{\text{in}})>b$, i.e., if the growth/birth kinetics dominate the rate of death, then for $D^*, X^*, S^*>0$ such that \eqref{eqpt-mu1} holds, we can apply Theorem \ref{Th1} and conclude that the feedback law \eqref{Feedback1} achieves global stabilization of $(X^*,S^*)$ (see Example \ref{Ex1} in Section \ref{Sec-Simulation}). 
    
Additionally, if $\mu(S_{\text{in}})<b$, meaning that the rate of death outpaces the rate of reproduction, Theorem \ref{Th2} indicates that, under the Haldane birth/growth kinetics, there is no positive-valued feedback law that globally stabilizes the model. 
    The only case that is not covered by assumption \textbf{(A)} is the measure-zero case $\mu(S_{\text{in}})=b$. The same conclusions hold for any other case where the kinetic equation for the specific growth rate is given by a $C^1$ function $\mu$ for which there exists  $\overline{S} \in (0, +\infty] $ with $\mu'(S)\geq 0$ for $S \in [0, \overline{S})$ and $\mu'(S)\leq 0$ for $S > \overline{S}$ and $\overline{S} < +\infty $ (a class of functions that contains all concave functions).  
    
    \section{An age-structured Chemostat Model}\label{Sec-Age}

    Consider the age-structured chemostat model: For $t\geq 0,$ $a>0$
    \begin{align}
        &\frac{\partial f}{\partial t}(t,a)+\frac{\partial f}{\partial a}(t,a)=-(\beta(a)+D(t))f(t,a),\label{system-age1}\\
        &f(t,0)=\mu(S(t))\int_0^{+\infty}k(a)f(t,a)da,\label{system-agei1}\\
        &\dot{S}(t)=D(t)(S_{\text{in}}-S(t))-\mu(S(t))\int_0^{+\infty}q(a)f(t,a)da,\label{system-ageS1}
    \end{align} where $f(t,a)>0$ is the distribution function of the microbial population in the chemostat at time $t\geq 0$ and age $a\geq 0,$ $S(t)>0$ is the limiting substrate concentration, $S_{\text{in}}>0$ is the inlet concentration of the substrate, $D(t)>0$ is the dilution rate, $\mu$ is the specific growth rate function, $\beta(a)$ is the mortality rate and $k, q$ are functions that determine the birth of new cells and the substrate consumption of the microbial population, respectively. All functions $\mu, \beta, k, q:\mathbb{R}_+\rightarrow\mathbb{R}_+$ are assumed to be bounded, $C^1(\mathbb{R}_+)$ functions with $\mu(0)=0,$ $\mu(S)>0$ for all $S>0$ and 
    \begin{align*}
        \int_0^{+\infty}k(a)da>0, \quad \int_0^{+\infty}q(a)da>0.
    \end{align*}
    Clearly, system \eqref{system-age1}-\eqref{system-agei1}-\eqref{system-ageS1} is a complicated nonlinear model that consists of a hyperbolic PDE with a non-local boundary condition and an ODE. Similar age-structured chemostat models have been studied in \cite{Toth2006-Limit}.

    In order to study model  \eqref{system-age1}-\eqref{system-agei1}-\eqref{system-ageS1}, we define for $t\geq 0:$
    \begin{align}
        X(t)=\int_0^{+\infty}q(a)f(t,a)da,\quad Y(t)=\int_0^{+\infty}k(a)f(t,a)da. \label{X(t)-Y(t)}
    \end{align}
    These variables $X$ and $Y$, respectively, refer to the {\em feed activity} (nutrient consumption activity) and {\em birth activity}. Assuming that the solution of \eqref{system-age1}-\eqref{system-agei1}-\eqref{system-ageS1} satisfies $\underset{a\rightarrow+\infty}{\lim} f(t,a)=0$ for all $t\geq 0$, meaning that the portion of population with infinite age is zero, we get using \eqref{system-age1}-\eqref{system-agei1}-\eqref{system-ageS1} and, \eqref{X(t)-Y(t)}:
    \begin{align}
        \dot{X}(t)=&q(0)\mu(S(t))Y(t)+\int_0^{+\infty}(q'(a)-\beta(a)q(a))f(t,a)da-D(t)X(t)\label{Xp}\\
        \dot{Y}(t)=&k(0)\mu(S(t))Y(t)+\int_0^{+\infty}(k'(a)-\beta(a)k(a))f(t,a)da-D(t)Y(t).\label{Yp}
    \end{align}
    We use the following assumption.\\ \\
    \textbf{(B)} There exist constants $b, \gamma, M\geq 0,$ $p_0,q_0,\overline{a}>0$ with $\beta(a)\leq b$ and $\gamma a\leq M \exp\left(ba-\int_0^{a}\beta(s)ds\right)$ for all $a\geq \overline{a}$ such that the following equations hold for $a\geq 0:$
    \begin{subequations}
    \begin{align}
        q(a)&=\exp\left(\int_0^{a}\beta(s)ds-ba\right)q_0,\\
        k(a)&=\exp\left(\int_0^a\beta(s)ds-ba\right)(p_0+\gamma q_0 a).
    \end{align}\label{q-k}
    \end{subequations}
    Assumptions similar to \textbf{(B)} have been used extensively in the literature (see \cite{Toth2006-Limit}). Notice that, since 
    \begin{align*}
        \beta(a)\leq b, \quad  \gamma a\leq M \exp\left(ba-\int_0^{a}\beta(s)ds\right), \quad \forall a\geq \overline{a},
    \end{align*} we deduce from \eqref{q-k} that $k,q$ are globally Lipschitz and bounded functions.

    Exploiting \eqref{Xp},\eqref{Yp}, \eqref{q-k} and \eqref{system-ageS1} we get the following ODE model:
    \begin{subequations}
        \begin{align}
        &\dot{S}=D(S_{\text{in}}-S)-\mu(S)X\\
        &\dot{X}=q_0\mu(S)Y-(b+D)X\\
        &\dot{Y}=p_0\mu(S)Y+\gamma X-(b+D)Y.
    \end{align}\label{systeme-ageODE}
    \end{subequations}
    We study next the ODE model \eqref{systeme-ageODE} with the physically meaningful state space
    \begin{align}
        (X,Y,S)\in (0,+\infty)^2\times (0, S_{\text{in}}).\label{space-age}
    \end{align}
    If the following holds
    \begin{align*}
        \frac{(b+D^*)^2}{p_0(b+D^*)+\gamma q_0}\in \mu((0,S_{\text{in}})),
    \end{align*} then the open-loop system \eqref{systeme-ageODE} with $D(t)\equiv D^*>0$ has equilibrium points. The open-loop equilibrium points for \eqref{systeme-ageODE} with $D(t)\equiv D^*>0$ can be found from the equations
    \begin{align}
\begin{array}{rcl}
\mu(S^*)&=&\displaystyle\frac{(b+D^*)^2}{p_0(b+D^*)+\gamma q_0},\quad
        \\
        X^*&=&\displaystyle\frac{D^*(S_{\text{in}}-S^*)}{(b+D^*)^2}(p_0(b+D^*)+\gamma q_0),\quad
        \\
        Y^*&=&\displaystyle\frac{D^*(S_{\text{in}}-S^*)}{q_0(b+D^*)^3}(p_0(b+D^*)+\gamma q_0)^2.
\end{array}
\label{eqpt2}
    \end{align}
    Just like the simple chemostat models, the above equations can give multiple equilibrium points, excluding the possibility of global asymptotic stability for any of them.

    The stability properties of equilibrium points can be found by checking the Jacobian matrix
    \begin{align*}
        \begin{pmatrix}
            &-D^*-\mu'(S^*)X^* \ &-\mu(S^*) \ &0\\
            &q_0\mu'(S^*)Y^* \ &-(b+D^*)\ &q_0\mu(S^*)\\
            &p_0\mu'(S^*)Y^* \ &\gamma \ &-b-D^*+p_0\mu(S^*)
        \end{pmatrix}.
    \end{align*}
    The characteristic polynomial of the Jacobian matrix is the cubic polynomial
    \begin{eqnarray}
        h(s)&=&s^3+\left(b+\frac{\gamma q_0(b+D^*)}{p_0(b+D^*)+\gamma q_0}+2D^*+\zeta\right)s^2+(b+D^*)\left(D^*+\frac{(D^*+\zeta)\gamma q_0}{p_0(b+D^*)+\gamma q_0}+2\zeta\right)s\nonumber\\
        &&+(b+D^*)^2\zeta,\label{h}
    \end{eqnarray} 
     where $\zeta=\mu'(S^*)X^*.$ It is clear that an equilibrium point of the open-loop system \eqref{systeme-ageODE} with $D(t)\equiv D^*>0$ is unstable when $\mu'(S^*)<0$.

    The problem of the global stabilization by means of a positive-valued feedback law (since $D(t)>0$) of an arbitrary equilibrium point of system \eqref{systeme-ageODE} with state space given by \eqref{space-age} is clearly non-trivial.

    Consider the transformation
    \begin{subequations}
        \begin{align}
        x_1=\ln\left(\frac{(S_{\text{in}}-S^*)X}{X^*(S_{\text{in}}-S)}\right), \quad x_2=\ln\left(\frac{X^*Y}{Y^*X}\right), x_3=\ln\left(\frac{S(S_{\text{in}}-S^*)}{S^*(S_{\text{in}}-S)}\right),
    \end{align}\label{x_1-x_2-x_3}
    \end{subequations}
    which maps $(0,+\infty)^2\times (0,S_{\text{in}})$ onto $\mathbb{R}^3$ and the point $(X^*,Y^*,S^*)\in (0,+\infty)^2\times (0,S_{\text{in}})$ to $0\in \mathbb{R}^3.$ Its inverse reads
    \begin{subequations}
        \begin{align}
        X=\frac{X^*(\kappa+1)}{p(x_3)}e^{x_1-x_3}, Y=\frac{Y^*(\kappa+1)}{p(x_3)}e^{x_1+x_2-x_3}, \quad S=\frac{S_{\text{in}}}{p(x_3)},
    \end{align}\label{X-Y-S}
    \end{subequations}
     where the constant $\kappa>0$ and the function $p$ are defined by $\eqref{g-p-kappa}$. Therefore, we rewrite the chemostat model \eqref{systeme-ageODE} as the following control system:
    \begin{subequations}
        \begin{align}
        &\dot{x}_1=(b+D^*)g(x_3)(e^{x_2}-1)+b(g(x_3)-1)+D^*g(x_3)(1-e^{x_1})\label{syst-age-x1}\\
        &\dot{x}_2=(b+D^*)(\lambda(1-g(x_3))+\lambda((e^{-x_2}-1)+g(x_3)(1-e^{x_2}))\label{syst-age-x2}\\
        &\dot{x}_3=p(x_3)(D-D^*g(x_3)e^{x_1})\\
        &x=(x_1,x_2,x_3)\in \mathbb{R}^3, D>0
    \end{align}\label{systeme-age-stud}
    \end{subequations}
     where the function $g$ is defined by \eqref{g-p-kappa} and the constant $\lambda\in [0,1)$ is given by
    \begin{align}
        \lambda:=\frac{\gamma X^*}{Y^*(b+D^*)}=1-\frac{p_0\mu(S^*)}{b+D^*}.\label{lambda}
    \end{align}
    As in the previous (lumped) case, in order to construct a locally Lipschitz, positive-valued stabilizing feedback for system \eqref{systeme-age-stud}, we need an assumption about the specific growth rate $\mu(S).$\\ \\
    \textbf{(C)} There exists a constant $\varphi>1$ such that the following inequality holds for $S\in [S^*, S_{\text{in}}]:$
    \begin{align}
        \left(b-\frac{\lambda(b+D^*)}{2}\right)\left(\frac{\mu(S^*)}{\mu(S)}-1\right)+(1+\lambda)\frac{\lambda^2\varphi(b+D^*)^2}{4D^*}\left(\frac{\mu(S^*)}{\mu(S)}-1\right)^2<D^*\left(1-\frac{1}{4(1+\lambda)\varphi}\right).\label{Assump-C}
    \end{align}
    Assumption \textbf{(C)} is a combined restriction on the lower bound of specific growth rate $\mu(S)$ for $S\in [S^*, S_{\text{in}}]$ and on the magnitude of the mortality rate $b$. When $\lambda=0$ (i.e when $\gamma=0;$ recall \eqref{lambda}), by using \eqref{eqpt2} we conclude that assumption \textbf{(C)} is equivalent to the requirement that the inequality $p_0\mu(S)>b$ holds for all $S\in [S^*, S_{\text{in}}].$ Therefore, in this case assumption \textbf{(C)} is identical to assumption \textbf{(A)}.

    We obtain the following result whose proof is given in Section \ref{Proof-Th3}.
    
    \begin{theorem}\label{Th3}
        Consider system \eqref{systeme-age-stud} under assumption \textbf{(C)}. Then for every $\delta>0$ and $\alpha\in [0,1),$ the feedback law
        \begin{align}
            D=(\kappa+1)D^*g(x_3)\frac{e^{x_1-x_3}}{p(x_3)}+\delta \begin{cases}
                |g(x_3)-1|^{1+\alpha},\quad  x_3\leq 0\\
                0\quad \text{if}\quad x_3>0
            \end{cases}\label{feedback2}
        \end{align} achieves global asymptotic stabilization of $0\in \mathbb{R}^3.$ Moreover, if $\alpha >0$ and $\lambda>0$ then the feedback law \eqref{feedback2} achieves in addition local exponential stabilization of $0\in \mathbb{R}^3$.
    \end{theorem}
    
    It should be noted that the feedback law \eqref{feedback2} coincides with the feedback law \eqref{Feedback1} with the controller gain $\delta b$ that is used in \eqref{Feedback1} replaced by $\delta.$ The feedback law \eqref{feedback2} in the original coordinates is given by equation \eqref{Feedback-origin} with the controller gain $\delta b$ that is used in \eqref{Feedback-origin} replaced by $\delta.$ The feedback law \eqref{feedback2} is completely independent of the state $Y.$

\section{Illustrative examples}\label{Sec-Simulation}

    In this section we present examples that illustrate the stabilization results of the previous sections.
    \begin{example}\label{Ex1}\rm
        We consider the chemostat model \eqref{system-simple1} with
        \begin{align*}
            \mu(S)=\frac{aS}{1+S+S^2}, \quad D^*=\frac{9}{10},\quad  b=\frac{1}{10}, \quad p_0=1, \quad  a=\frac{7}{2}, \quad S_{\text{in}}=2+\frac{10}{3}\approx 5.333.
        \end{align*}
        The system has two equilibrium points, namely the points
        \begin{align*}
            (X^*,S^*)=(3,2) \quad \text{and}\quad (X^{**},S^{**})=\left(\frac{87}{20},\frac{1}{2}\right)
        \end{align*}
        Since $\mu'(S^*)<0$ and $\mu'(S^{**})>0,$ it holds from Fact \ref{fact1} that the point $(X^*,S^*)$ is unstable while the point $(X^{**},S^{**})$ is locally asymptotically stable. The phase diagram of the open-loop system is shown in Figure \ref{fig-open-loop}. It becomes clear from Figure \ref{fig-open-loop} that the stable manifold of the point $(X^*,S^*)$ divides the state space into two regions: one region (left of the stable manifolds) where all solutions are attracted by the washout equilibrium point $(X,S)=(0,S_{\text{in}})$ and one region (right of the stable manifold) where all solutions are attracted by the stable equilibrium point $(X^{**},S^{**}).$
        
        \begin{figure}[t]
            \centering
            \includegraphics[width=0.5\linewidth]{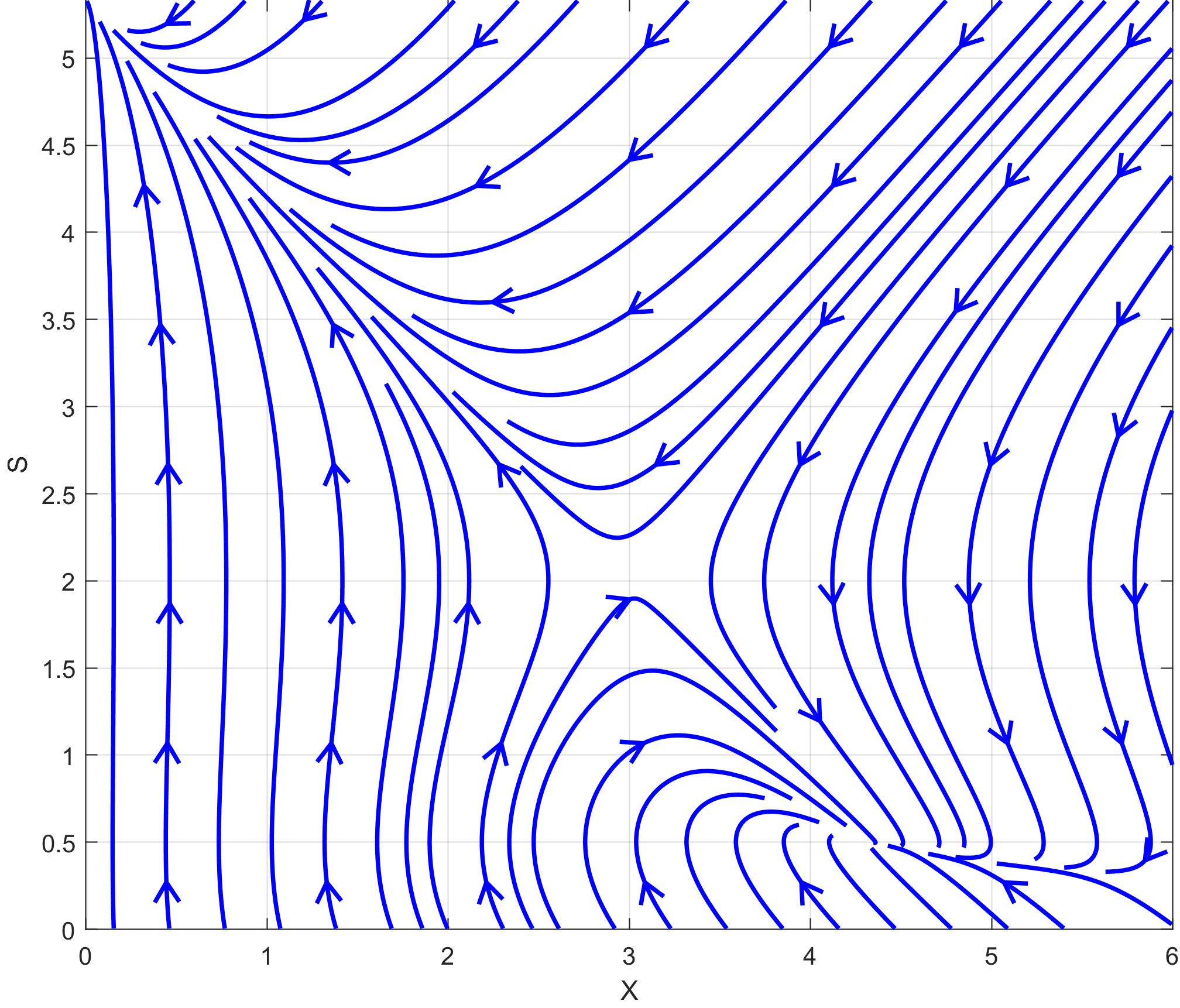}
            \caption{The phase diagram for the open-loop chemostat model \eqref{system-simple1} with $D=D^*=\frac{9}{10}$.}
            \label{fig-open-loop}
        \end{figure}
        
Noticing that
    \begin{align*}
        p_0\mu(S_{\text{in}})=\frac{42}{79}>b,
    \end{align*}
    the system \eqref{system-simple1} is globally stabilizable by the feedback law \eqref{Feedback-origin}. Figure \ref{fig-closed-loop} shows the phase diagram of the chemostat model \eqref{system-simple1} in closed loop with \eqref{Feedback-origin} for $\delta=10$ and  $\alpha=0.5$: as anticipated, all solutions are attracted by the point $(X^*,S^*).$

    \begin{figure}[t]
        \centering
        \includegraphics[width=0.5\linewidth]{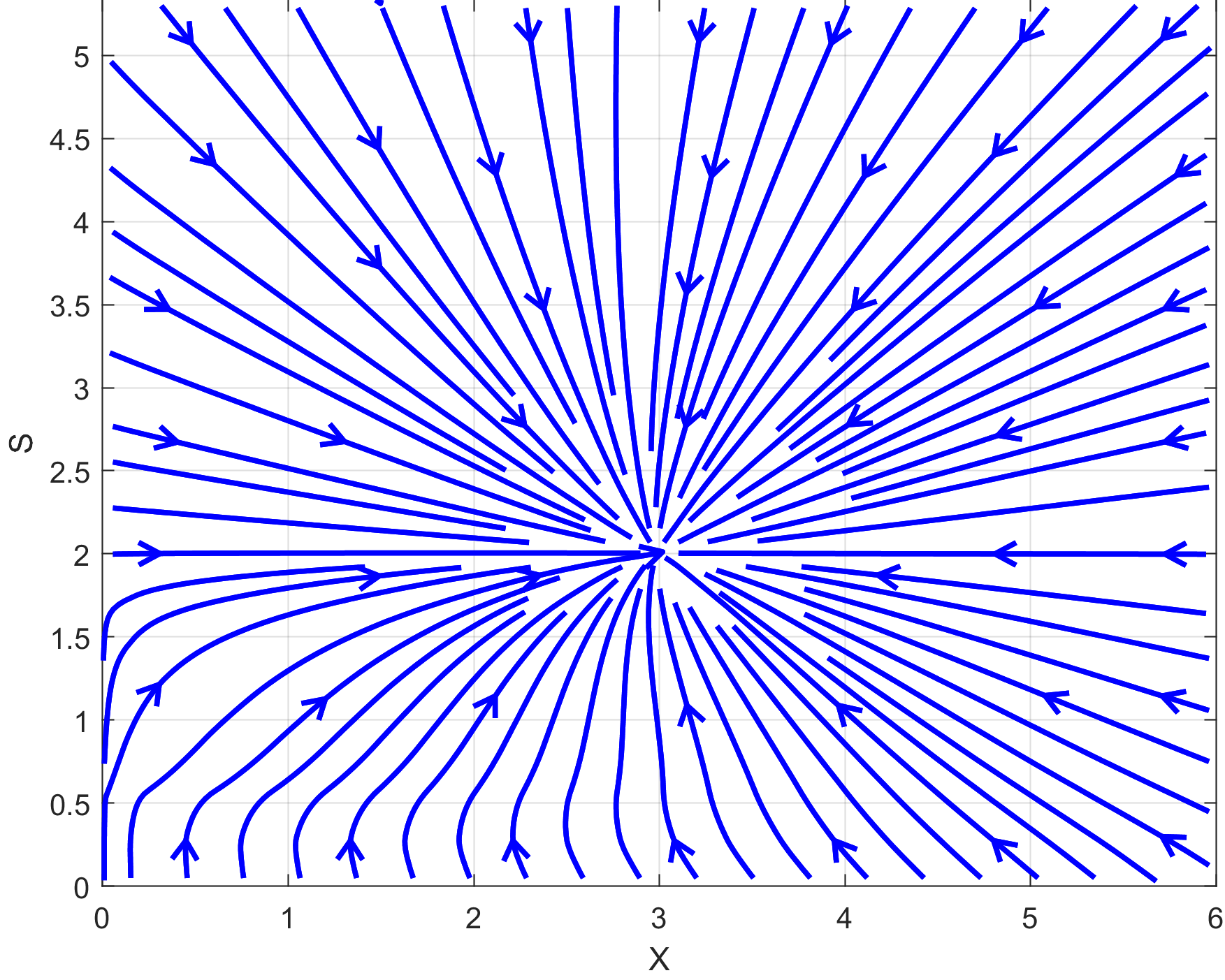}
        \caption{The phase diagram for the closed-loop chemostat model \eqref{system-simple1} with \eqref{Feedback-origin}.}
        \label{fig-closed-loop}
    \end{figure}

        \begin{figure}[b]
        \centering
        \includegraphics[width=0.5\linewidth]{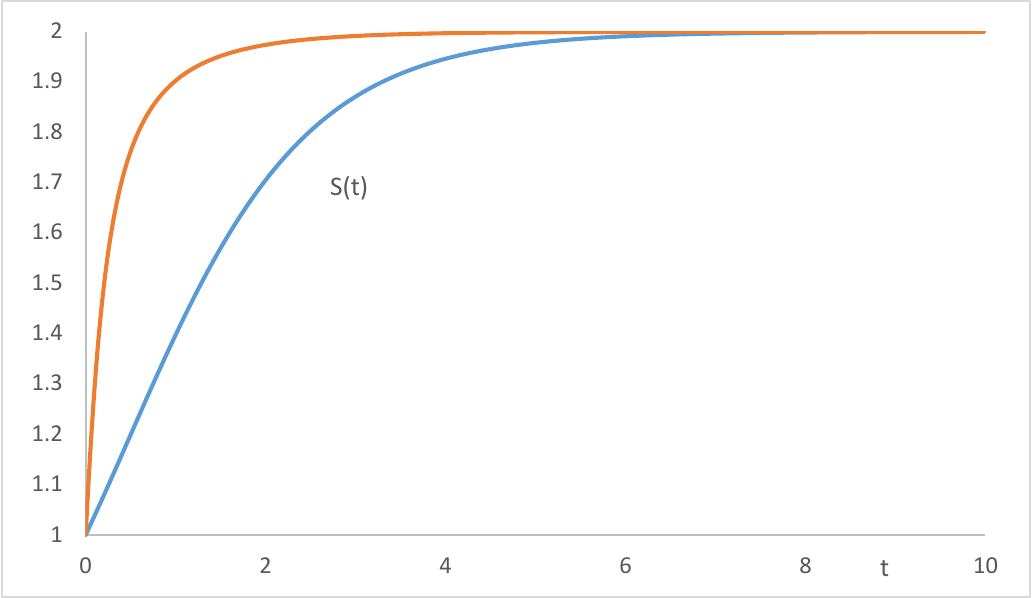}
        \caption{The component $S(t)$ of the solution of the closed-loop system \eqref{system-simple1} with \eqref{Feedback-origin}, $\alpha=0.5$ and initial condition $(X(0),S(0))=(1,1).$ Red line: $\delta=100.$ Blue line: $\delta=1.$}
        \label{fig-S(t)}
    \end{figure}
    
    Equations \eqref{system-simple1}, \eqref{eqpt-mu1} and \eqref{Feedback-origin} show that the following ODE is valid for the closed-loop chemostat model \eqref{system-simple1} with  \eqref{Feedback-origin}:
    \begin{align}
        \dot{S}=\mu(S)X\frac{S^*-S}{S_{\text{in}}-S^*}+\frac{\delta b(S_{\text{in}}-S)}{(\mu(S^*))^{1+\alpha}}\begin{cases}
            |\mu(S)-\mu(S^*)|^{1+\alpha}, \quad \text{if}\quad S\leq S^*\\
            0, \quad \text{if}\quad S>S^*.
        \end{cases}\label{Sdot}
    \end{align}
    Differential equation \eqref{Sdot} shows that the regions
    \begin{align*}
        \{(X,S)\in \mathbb{R}^2:\ S^*<S<S_{\text{in}}, \quad  X>0\},\quad\{(X,S^*)\in \mathbb{R}^2: \ X>0\}, \quad \text{and}\\ 
        \{(X,S)\in \mathbb{R}^2:\ 0<S<S^*, \ X>0\}
    \end{align*} are positively invariant regions for the closed-loop chemostat model \eqref{system-simple1} with \eqref{Feedback-origin}. Therefore, the controller gain $\delta>0$ has no effect on the solutions of the closed-loop system \eqref{system-simple1} with \eqref{Feedback-origin} when $S\in [S^*,S_{\text{in}}).$ On the other hand, higher values of the controller gain $\delta>0$ accelerate the convergence of the closed-loop chemostat model \eqref{system-simple1} with \eqref{Feedback-origin} to the equilibrium point $(X^*,S^*)$ when $S(0)<S^*$. This is shown in Figure \ref{fig-S(t)}.
\end{example}
\begin{example}\label{Ex2}\rm
    We consider the chemostat model \eqref{systeme-ageODE} with 
    \begin{align*}
        \mu(S)=\frac{aS}{1+S+S^2}, \quad D^*=\frac{9}{10}, \quad b=\frac{1}{10}, \quad p_0=0.8,
        q_0=1, \quad \gamma=0.2,\quad   a=\frac{7}{2}, \\ S_{\text{in}}=2+\frac{10}{3}\approx 5.333.
    \end{align*}
    The system has two equilibrium points, namely the points
    \begin{align*}
        (X^*,Y^*,S^*)=(3,3,2), \quad  (X^{**},Y^{**},S^{**})=\left(\frac{87}{20},\frac{87}{20}, \frac{1}{2}\right).
    \end{align*}
    Using \eqref{h}, we can determine the characteristic polynomials of the Jacobian matrices evaluated at the equilibrium points:

    Characteristic polynomial at $(X^*,Y^*,S^*):$ $h(s)=s^3-2.4s^2-8.82s-4.5$

    Characteristic polynomial at $(X^{**},Y^{**},S^{**}):$ $h(s)=s^3+6.6s^2+10.98s+4.5$.

    Therefore, the equilibrium point $(X^*,Y^*,S^*)$ is unstable with one positive eigenvalue (and two negative eigenvalues) while the equilibrium point $(X^{**},Y^{**},S^{**})$ is locally asymptotically stable with three real negative eigenvalues. Unless a solution starts on the stable manifold of $(X^*,Y^*,S^*),$ all other solutions converge either to $(X^{**},Y^{**},S^{**})$ or to the washout equilibrium $(X,Y,S)=(0,0,S_{\text{in}}).$ This is depicted in Figure \ref{fig-age-open}.
    
    \begin{figure}[t]
        \centering
        \includegraphics[width=0.5\linewidth]{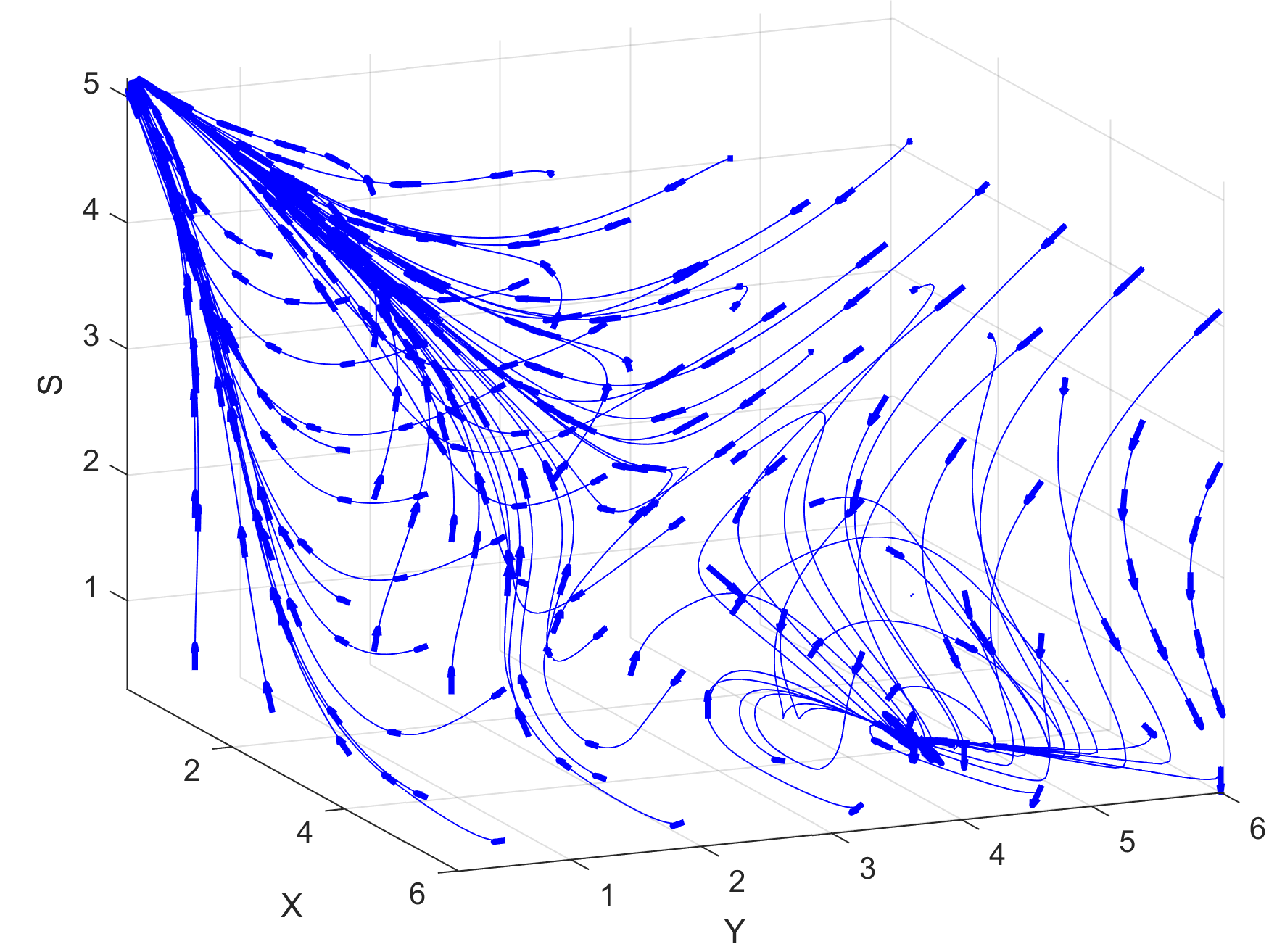}
        \caption{The phase diagram for the chemostat model \eqref{systeme-ageODE}.}
        \label{fig-age-open}
    \end{figure}
    
    Clearly, there is need for control for the stabilization of the point $(X^*,Y^*,S^*).$ We have determined by numerical means that assumption \textbf{(C)} is valid (with $\varphi=1.1)$ for system \eqref{systeme-ageODE}. Therefore, Theorem \ref{Th3} guarantees that for any $\alpha\in [0,1),$ the feedback law
    \begin{align}
        D=\frac{D^*\mu(S)X}{\mu(S^*)X^*}+\frac{\delta}{(\mu(S^*))^{1+\alpha}}\begin{cases}
            |\mu(S)-\mu(S^*)|^{1+\alpha},\quad \text{if}\quad S\leq S^*\\
            0, \quad \text{if}\quad S>S^*
        \end{cases}\label{D-example}
    \end{align} achieves global asymptotic stabilization of the equilibrium point $(X^*,Y^*,S^*).$ This is shown in Figure \ref{fig-age-closed}, where the phase diagram of chemostat model \eqref{systeme-ageODE} in closed loop with \eqref{D-example} with $\delta=1$ and $\alpha=0.5$ is shown: all solutions are attracted by the equilibrium point $(X^*,Y^*,S^*)$ .
    \begin{figure}[t]
        \centering
        \includegraphics[width=0.5\linewidth]{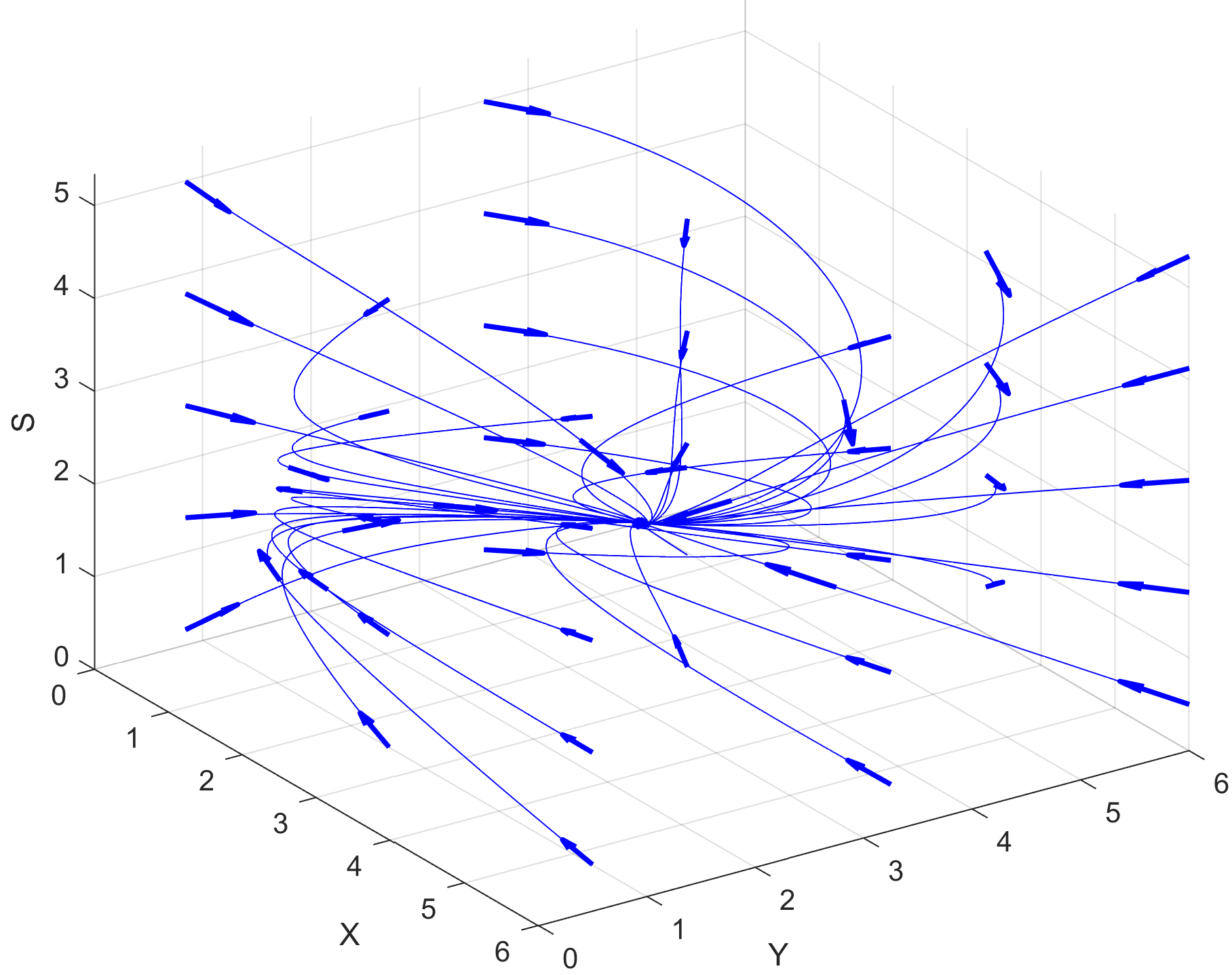}
        \caption{The phase diagram of the closed-loop system \eqref{systeme-ageODE} with \eqref{D-example}}and $\delta=1, \ \alpha=0.5$
        \label{fig-age-closed}
    \end{figure}
\end{example}

    \section{Proofs}\label{Sec-Proofs}
    We first consider the following claim which will be useful in the proofs of our main results. It follows from \eqref{X-S-1}, \eqref{g} and the fact that $\mu$ is $C^1$ on $[0,S_{\text{in}}].$ Its proof is given in Appendix \ref{Claim-Lipschitz}.
\begin{claim}\label{Claim-Lipschitz}
    There exists $A\geq 0$ such that
    \begin{align}
        |g(z)-1|\leq A|e^{z}-1|, \quad \forall z\in \mathbb{R}.\label{Glipscitz}
    \end{align}
\end{claim} 


    \subsection{Proof of Theorem \ref{Th1}}\label{Proof-Th1}
The proof is made using Lyapunov function. Let $\alpha\in [0,1)$ and $\delta>0$. We define for all $x\in \mathbb{R}^2$ the continuously differentiable function:
    \begin{align}
        V(x)=e^{x_1}-x_1-1+Q(x_2),\label{V(x)}
    \end{align}
where $Q$ is the function defined by
    \begin{align}
        Q(x_2):=\begin{cases}
            \frac{1}{2}x_2^2+\frac{b}{2D^*\delta}\int_{x_2}^0\frac{|g(s)-1|^{1-\alpha}}{p(s)g(s)}ds, \quad \text{if}\quad x_2\leq 0\\ \\
            R\int_0^{x_2}\frac{e^s}{g(s)^2}(e^{s}-1)ds,\quad \text{if}\quad x_2>0,
        \end{cases}\label{Q}
    \end{align} and $R$ a positive constant that will be selected later on.
        We claim that the function $Q$ given by \eqref{Q} is continuously differentiable, positive definite and  radially unbounded. Indeed, the facts that $\alpha \in [0,1)$ and $g(0)=1$ guarantee that $Q$ is differentiable at $0$ with $Q'(0)=0$. Thus the function $Q$ is differentiable on $\mathbb{R}$ with
    \begin{align}
        Q'(x_2)=\begin{cases}
            x_2-\frac{b}{2D^*\delta}\frac{|g(x_2)-1|^{1-\alpha}}{p(x_2)g(x_2)},\quad \text{if}\quad x_2\leq 0\\ \\
            \frac{R e^{x_2}}{g(x_2)^2}(e^{x_2}-1), \quad \text{if} \quad x_2>0.
        \end{cases}\label{Q'}
    \end{align}
    Moreover, since $\mu$ is positive and bounded, so is the function $g$ (see \eqref{g-p-kappa}). Consequently, the quantity $\overline{g}:=\underset{s\in \mathbb{R}}{\sup}\ g(s)>0$ is a positive constant and we obtain from \eqref{Q} and the facts that $R>0$ and $p(x_2)>0$ that
    \begin{align*}
        \lim_{x_2\rightarrow-\infty} Q(x_2)=\lim_{x_2\rightarrow-\infty}\frac{1}{2}x_2^2+\frac{b}{2D^*\delta}\int_{x_2}^0\frac{|g(s)-1|^{1-\alpha}}{p(s)g(s)}ds\geq \lim_{x_2\rightarrow-\infty}\frac{1}{2}x_2^2 =+\infty,   \\ 
        \lim_{x_2\rightarrow+\infty} Q(x_2)=\lim_{x_2\rightarrow+\infty}R\int_0^{x_2}\frac{e^s}{g(s)^2}(e^{s}-1)ds\geq\frac{R}{\overline{g}^2}\int_0^{+\infty}(e^s-1)ds=+\infty.
    \end{align*}
   which ensures that the function $Q$ is radially unbounded. 
    
    Since the function $x_1\mapsto e^{x_1}-x_1-1$ is positive definite and radially unbounded and since  $Q$ is positive definite and radially unbounded, it follows that the function $V$ defined by \eqref{V(x)} is positive definite and radially unbounded.

We recall that from \eqref{systeme-simple2}, \eqref{g-p-kappa} and \eqref{Feedback1} we have the following equations
    \begin{eqnarray}
        D(x)&=&(\kappa+1)D^*g(x_2)\frac{e^{x_1-x_2}}{p(x_2)}+u(x_2)\label{Dproof}\\
        \dot{x}_2 &=& D^*g(x_2)e^{x_1-x_2}(1-e^{x_2})+p(x_2)u(x_2)\,,\label{x_2p}
    \end{eqnarray} where
    
    \begin{align}
        u(x_2):=\begin{cases}
            \delta b|g(x_2)-1|^{1+\alpha},\quad \text{if}\quad x_2\leq 0\\ 
            0\quad \text{if}\quad x_2>0
        \end{cases}\label{u}
    \end{align} 

    From \eqref{systeme-simple2}, \eqref{x_2p} and \eqref{V(x)} we have for all $x\in \mathbb{R}^2:$
    \begin{eqnarray}
        \dot{V}(x)&=&b(g(x_2)-1)(e^{x_1}-1)-D^*g(x_2)(e^{x_1}-1)^2+Q'(x_2)p(x_2)u(x_2)\nonumber\\
        &&-D^*Q'(x_2)g(x_2)e^{x_1-x_2}(e^{x_2}-1).\label{Vpx1}
    \end{eqnarray}
    In order to make $\dot{V}$ negative, let us notice that in view of the expression \eqref{Vpx1} we can 
    \begin{itemize}
        \item 
    first set $u$ to be $0$ for the positive value of $x_2$ and make $Q'$ positive enough to compensate the first term in $\dot{V}.$
    \item and for the non-positive value of $x_2,$ we can select the function $Q$ such as its derivative decrease enough to compensate the first term in $\dot{V}.$
    \end{itemize}
    That is how we construct the above functions $Q$ and $u.$
    Using the following Young inequality
    \begin{align*}
        b(g(x_2)-1)(e^{x_1}-1)\leq \frac{b^2}{2D^*g(x_2)}(g(x_2)-1)^2
        +\frac{D^*g(x_2)}{2}(e^{x_1}-1)^2,
    \end{align*} we obtain from \eqref{Vpx1} that
    \begin{eqnarray}
        \dot{V}(x)&\leq& \frac{b^2}{2D^*g(x_2)}(g(x_2)-1)^2-\frac{D^*g(x_2)}{2}(e^{x_1}-1)^2+Q'(x_2)p(x_2)u(x_2)\nonumber\\
        &&-D^*Q'(x_2)g(x_2)e^{x_1-x_2}(e^{x_2}-1).\label{Vpx2}
    \end{eqnarray}
    We can notice that the presence of $x_1$ in the last term of \eqref{Vpx2} imposes some discussions about the growth of $x_1$ in order to exploit a potential negativity from the last term to compensate the first one. That is where we invoke assumption \textbf{(A)} to get the following
    \begin{claim}\label{Claim-Th1}
        If assumption \textbf{(A)} holds, then there exists $r\in [0,1)$ such that
        \begin{align}
    b\frac{g(x_2)-1}{g(x_2)}\geq -rD^*, \quad \forall x_2\geq 0.\label{AssumpA-equiv}
\end{align}
    \end{claim}

Its proof is given in Appendix \ref{Claim-Th1}.
    
    Let $r\in [0,1)$ be the constant for which \eqref{AssumpA-equiv} holds. Since $r\in [0,1),$ there exists a constant $c>0$ such that 
    \begin{align}
        r+e^{-c}<1.\label{r-ce}
    \end{align} 
We distinguish the following cases 

\begin{figure}[t]
    \centering
\begin{tikzpicture}[line cap=round,line join=round,>=triangle 45,x=1.0cm,y=1.0cm]
\draw[->,color=black] (-0.7,0) -- (8.0,0);
\foreach \x in {-1,1,2,3,4,5,6,7,8,9,10}
\draw[shift={(\x,0)},color=black] (0pt,-2pt);
\draw[->,color=black] (0,-0.5) -- (0,10.71);
\clip(-1.02,-1.05) rectangle (8.0,11.71);
\draw (-0.70,1.80) node[anchor=north west] {$ S^* $};
\draw (-0.79,10.21) node[anchor=north west] {$ S_{\text{in}} $};
\draw (3.70,10.50) node[anchor=north west] {$ \textcolor{purple}{\text{Case 1}} $};
\draw (0.01,8.10) node[anchor=north west] {$ \textcolor{red}{\text{Case 2} } $};
\draw (-0.70,11.0) node[anchor=north west] {$ S $};
\draw (-0.38,0) node[anchor=north west] {$ 0 $};
\draw (7.5,0.03) node[anchor=north west] {$ X $};
\draw (3.7,9.84) node[anchor=north west] {\parbox{4.09 cm}{Current substrate is both in $\textcolor{purple}{\text{excess}}$ of the  equilibrium level, $S>S^*$ and provides $\textcolor{purple}{\text{sufficient nurishnment}}$ even if the current biomass $X$ exhibits underpopulation $S\geq  S_{\text{in}}-\frac{\mu(S^*)e^c}{D^*}X$}};
\draw (0.07,7.46) node[anchor=north west] {\parbox{3.36 cm}{Current\\ substrate \\ is in $\textcolor{red}{\text{excess}}$\\ of the  equili-\\brium level $S^*$\\  but $\textcolor{red}{\text{undernuri-}}$\\ $\textcolor{red}{\text{shes}}$ the currently underpopulated\\ biomass $X,$\\ $S<  S_{\text{in}}-\frac{\mu(S^*)e^c}{D^*}X$}};
\draw (0.7,1.3) node[anchor=north west] {\parbox{6.5 cm}{$ \textcolor{blue}{\text{Case 3}} $ Biomass is $\textcolor{blue}{\text{underfed}}$ relative to the equilibrium  $S^*,$ $S\leq S^*$ }};
\draw (0,10)-- (5.7,1.38);
\draw (0,1.3)-- (9.75,1.35);
\end{tikzpicture}
\caption{Physical interpretations of Cases $1,$ $2,$ $3.$}
    \label{fig:physical-int}
\end{figure}
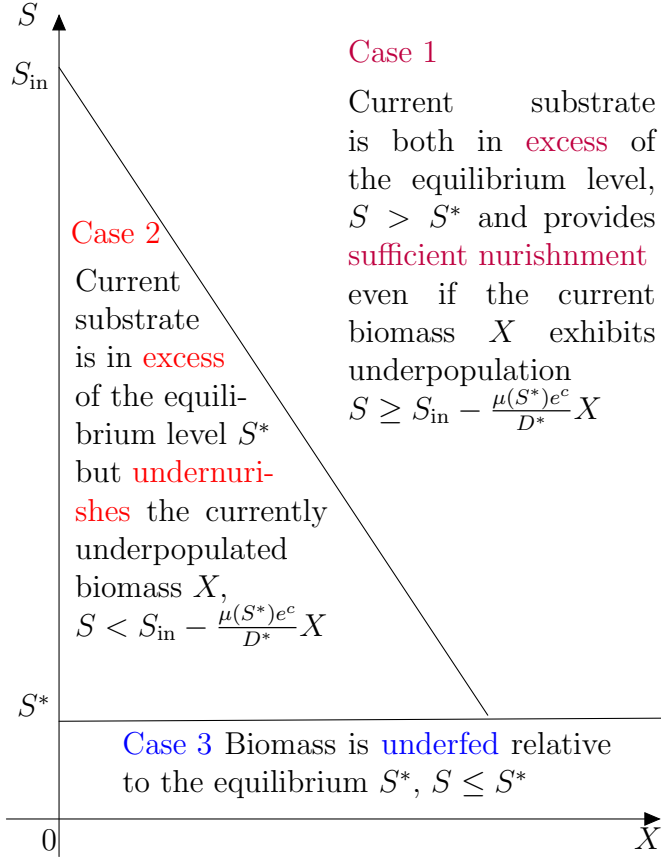  

\begin{enumerate}
\item[]Case $1:$ $x_2>0$ and $x_1\geq -c.$ In view of \eqref{x_1-x_2}, this case corresponds to the current substrate being both {\em in excess of the equilibrium level}, $S>S^*$, and {\em providing sufficient nurishnment even if the current biomass $X$ exhibits underpopulation}, $S\geq  S_{\text{in}}-\frac{\mu(S^*)e^c}{D^*}X$. 
In this case, in view of  \eqref{u} and \eqref{Vpx2} we get:
        \begin{align*}
            \dot{V}(x)\leq \frac{b^2}{2D^*g(x_2)}(g(x_2)-1)^2-\frac{D^*g(x_2)}{2}(e^{x_1}-1)^2
            -D^*Q'(x_2)g(x_2)e^{x_1-x_2}(e^{x_2}-1).
        \end{align*}
        Thus, from \eqref{Q'} it holds that
        \begin{align*}
            \dot{V}(x)\leq \frac{b^2}{2D^*g(x_2)}(g(x_2)-1)^2-\frac{D^*g(x_2)}{2}(e^{x_1}-1)^2
            -\frac{RD^*}{g(x_2)}e^{x_1}(e^{x_2}-1)^2.
        \end{align*}
        So, using \eqref{Glipscitz} and the fact that $x_1\geq -c$ we obtain
        \begin{align*}
            \dot{V}(x)
            \leq \frac{A^2b^2}{2D^*g(x_2)}(e^{x_2}-1)^2-\frac{D^*g(x_2)}{2}(e^{x_1}-1)^2
            -\frac{RD^*}{g(x_2)}e^{-c}(e^{x_2}-1)^2.
        \end{align*}
        By selecting the constant $R$ large enough that  $R>\frac{A^2b^2}{{D^*}^2}e^{c},$ it holds that
        \begin{align}
            \dot{V}(x)\leq -\frac{D^*g(x_2)}{2}(e^{x_1}-1)^2-\frac{RD^*}{2g(x_2)}e^{-c}(e^{x_2}-1)^2.\label{Case2}
        \end{align}
        \item[]Case $2:$ $x_2>0$ and $x_1<-c.$ In the original coordinates, this case corresponds to the current substrate being {\em in excess of the equilibrium level}, $S>S^*$, but nevertheless {\em undernurishing the currently underpopulated biomass $X$}, $S<  S_{\text{in}}-\frac{\mu(S^*)e^c}{D^*}X$
In this case it holds that $D^*(e^{x_1}-1)<D^*(e^{-c}-1)$, and consequently
        \begin{align*}
            b\frac{g(x_2)-1}{g(x_2)}-D^*(e^{x_1}-1)>b\frac{g(x_2)-1}{g(x_2)}-D^*(e^{-c}-1).
        \end{align*}
        By virtue of \eqref{AssumpA-equiv} and \eqref{r-ce}, we get
        \begin{align}
            b\frac{g(x_2)-1}{g(x_2)}-D^*(e^{x_1}-1)>D^*(1-r-e^{-c})>0.\label{bg}
        \end{align}
        We have from \eqref{u} and \eqref{Vpx1} that 
        \begin{align*}
          \dot{V}(x)=b(g(x_2)-1)(e^{x_1}-1)-D^*g(x_2)(e^{x_1}-1)^2
          -D^*Q'(x_2)g(x_2)e^{x_1-x_2}(e^{x_2}-1).
        \end{align*}
        In view of the expression \eqref{Q'}, it holds that
        \begin{eqnarray*}
            \dot{V}(x)&=&b(g(x_2)-1)(e^{x_1}-1)-D^*g(x_2)(e^{x_1}-1)^2-\frac{RD^*}{g(x_2)}e^{x_1}(e^{x_2}-1)^2\\
            &=&-g(x_2)\left[b\frac{(g(x_2)-1)}{g(x_2)}-D^*(e^{x_1}-1)\right](1-e^{x_1})-\frac{RD^*}{g(x_2)}e^{x_1}(e^{x_2}-1)^2.
        \end{eqnarray*}
        This combining with \eqref{bg} and the fact that $x_1<-c<0,$ ensure that
        \begin{align}
            \dot{V}(x)
            &\leq -g(x_2)D^*(1-r-e^{-c})(1-e^{x_1})-\frac{RD^*}{g(x_2)}e^{x_1}(e^{x_2}-1)^2.\label{Case3}
        \end{align}
    
        \item[]Case $3:$ $x_2\leq 0.$ This is equivalent to $S\leq S^*$ and it means that biomass is {\em underfed} relative to the equilibrium $S^*$. In this case we obtain from \eqref{u} and \eqref{Vpx2}:
        \begin{align*}
            \dot{V}(x)&\leq \frac{b^2}{2D^*g(x_2)}(g(x_2)-1)^2-\frac{D^*g(x_2)}{2}(e^{x_1}-1)^2-Q'(x_2)D^*g(x_2)e^{x_1-x_2}(e^{x_2}-1)\\
            &+\delta bQ'(x_2)p(x_2)|g(x_2)-1|^{1+\alpha}.
        \end{align*}
        And using \eqref{Q'}, this yields the following estimate
        \begin{eqnarray*}
            \dot{V}(x)&\leq& -\frac{D^*g(x_2)}{2}(e^{x_1}-1)^2-x_2D^*g(x_2)e^{x_1-x_2}(e^{x_2}-1)\\
            &&+\frac{b}{2\delta}\frac{|g(x_2)-1|^{1-\alpha}}{p(x_2)}e^{x_1-x_2}(e^{x_2}-1)+\delta b x_2p(x_2)|g(x_2)-1|^{1+\alpha}.
        \end{eqnarray*}
        Since $x_2\leq 0,$ and $p(x_2)>0$ we obtain
        \begin{align}
            \dot{V}(x)&\leq -\frac{D^*g(x_2)}{2}(e^{x_1}-1)^2-|x_2|D^*g(x_2)e^{x_1} |e^{-x_2}-1|.\label{Case1}
        \end{align}
        
    \end{enumerate}
    Combining all cases, we conclude from \eqref{Case2}, \eqref{Case3}  and \eqref{Case1}  that $\dot{V}(x)<0$ for all $x\in \mathbb{R}^2\setminus\{0\}.$ Therefore, by Lyapunov's theorem $0\in \mathbb{R}^2$ is globally asymptotically stable for the closed-loop system \eqref{systeme-simple2} with \eqref{Feedback1}. 

    Let now  $\alpha\in (0,1)$. Since $\alpha>0$ and $g \in C^1(\mathbb{R})$, $\ g(0)=1$, it follows that the function $u$ defined by \eqref{u} is differentiable on $\mathbb{R}$ and its derivative is given by the following continuous function
    \begin{align}
       u'(x_2)=
       \delta b \begin{cases}
            (1+\alpha) g'(x_2)\text{sgn}(g(x_2)-1)|g(x_2)-1|^{\alpha}, \ x_2\leq 0\\
            0,\ x_2>0
        \end{cases}\label{u'} 
    \end{align} where 
    \begin{align*}
    \text{sgn}(g(x_2)-1)=\begin{cases}
        -1, \ g(x_2)\leq 1\\
        1, \ g(x_2)>1.
    \end{cases}
    \end{align*} It follows that the vector field of the closed-loop system \eqref{x_1-simple}-\eqref{x_2p} is continuously differentiable on $\mathbb{R}^2$ and then its local exponential stability at the origin $0\in \mathbb{R}^2$ is guaranteed by the exponential stability of its linearization at $0\in \mathbb{R}^2$.  Using \eqref{g}, \eqref{p}, \eqref{u} and \eqref{u'}, the Jacobian matrix of the closed-loop system \eqref{x_1-simple}-\eqref{x_2p} at $0\in \mathbb{R}^2$ is given by \begin{align*}
        \begin{pmatrix}
            &-D^*\quad &bg'(0)\\
            &0\quad &-D^*
        \end{pmatrix}.
    \end{align*}
    This matrix is Hurwitz with a unique eigenvalue $-D^*<0.$ So the linearization of the closed-loop system \eqref{x_1-simple}-\eqref{x_2p} at $0\in \mathbb{R}^2$ is exponentially stable. Therefore when $\alpha>0$, the feedback law \eqref{Feedback1} achieves local exponential stabilization of the system \eqref{systeme-simple2} at  $0\in \mathbb{R}^2$. 
      
    There is only left to prove that the proposed feedback law is locally Lipschitz. To that aim, recall that the feedback law \eqref{Feedback1} can be written as 
    \begin{align*}
        D(x)=(\kappa+1)D^*g(x_2)\frac{e^{x_1-x_2}}{p(x_2)}+u(x_2),
    \end{align*}
    where the term $u(x_2)$ was defined in \eqref{u}. In view of the assumed regularity on $p$ and $g$, and recalling that $p(x_2)>1$, the first term in this expression is continuously differentiable, hence locally Lipschitz. Moreover, the function $u$ defined in \eqref{u} is a continuous concatenation of two continuously differentiable functions and, as such, is locally Lipschitz as well. Thus, the feedback law $D(x)$ is locally Lipschitz, as the sum of two locally Lipschitz functions.

    \subsection{Proof of Theorem \ref{Th2}}\label{Proof-Th2}
    The proof is made by contradiction. Suppose that there exists a locally Lipschitz feedback law $D(x)\geq 0$ that achieves boundedness of solutions for system \eqref{systeme-simple2}. Define 
    \begin{align}
        \theta:=b-p_0\mu(\overline{S}), \label{theta-def}
    \end{align}
    and notice that, thanks to the assumption \eqref{AssumpTh2}, $\theta>0$. The equality \eqref{theta-def} in conjunction with \eqref{eqpt-mu1} and \eqref{g-p-kappa} guarantee the existence of a constant $\beta>0$ such that
    \begin{align}
        (b+D^*)g(\beta)=b-\theta, \label{bDg}
    \end{align} 
    and the assumption \eqref{AssumpTh2*} with \eqref{eqpt-mu1} and \eqref{g-p-kappa} ensure that
    \begin{align}
        g'(x_2)\leq 0, \quad \forall x_2\geq \beta,\label{g-decrease}
    \end{align}
    which in turn guarantees that
    \begin{align}
        g(x_2)\leq g(\beta),\quad  \forall x_2\geq \beta\label{bDx_21}
    \end{align} and thus
    \begin{align}
        (b+D^*)g(x_2)\leq (b+D^*)g(\beta)\quad  \forall x_2\geq \beta.\label{bDx_2}
    \end{align}
    Since by assumption $\mu(S_{\text{in}})<\mu(\overline{S})$ (making $S_{\text{in}}\neq \overline{S}$: see \eqref{AssumpTh2*}), there exists in particular $ \overline{x}_2>\beta$ such that $g(\overline{x}_2)< g(\beta)$, meaning that
    \begin{align}
        (b+D^*)g(\overline{x}_2)< (b+D^*)g(\beta).\label{x2beta}
    \end{align} 
    In view of \eqref{bDg}, the inequality \eqref{x2beta} yields
    \begin{align}
        \theta+(b+D^*)g(\overline{x}_2)<b.\label{assp-theta}
    \end{align}
    Moreover, combining \eqref{bDx_2} with the equality \eqref{bDg} we have $
        (b+D^*)g(x_2)\leq b-\theta$ for all $x_2\geq \beta$.
    This shows that, for all $x_2\geq \beta$,
    \begin{align}
        b(g(x_2)-1)\leq -(\theta+D^*g(x_2)).\label{bD*}
    \end{align}
    Thus, we have from \eqref{x_1-simple} and \eqref{bD*} that for all $x=(x_1,x_2)\in \mathbb{R}^2$ with $x_2\geq \beta$,
    \begin{align*}
        \dot{x}_1&=b(g(x_2)-1)+D^*g(x_2)(1-e^{x_1})\leq -\theta-D^*g(x_2)e^{x_1}.
    \end{align*}
    Since $D^*g(x_2) e^{x_1}\geq 0,$ we obtain for all $(x_1,x_2)\in \mathbb{R}^2$ with $x_2\geq \beta:$
    \begin{align}
       \dot{x}_1 \leq -\theta.\label{x_1p-theta}
    \end{align}
    Using \eqref{x_2-simple}, \eqref{g-decrease} and the fact that $D\geq 0,$ we obtain for all $x\in \mathbb{R}^2$ with $x_2\geq \beta:$
    \begin{align}
        g'(x_2)\dot x_2&=D g'(x_2)p(x_2)-g'(x_2)p(x_2)D^*g(x_2)e^{x_1}\nonumber\\
        &\leq |g'(x_2)|p(x_2)D^*g(x_2)e^{x_1}\label{g'p*}
    \end{align}
    Thus using \eqref{g-p-kappa}, we have $p(x_2)\leq (\kappa+1)$ for all $x_2\geq 0,$ and we get from \eqref{g'p*} that for all $(x_1,x_2)\in \mathbb{R}^2$ with $x_2\geq \beta$
    \begin{align}
        g'(x_2)\dot x_2
        \leq D^*(\kappa+1)\max\left\{|g'(x_2)|g(x_2), \ x_2\geq \beta\right\}e^{x_1}\label{g'1}
    \end{align}
    Since from \eqref{bDx_21} and \eqref{bDg}, we have that
    \begin{align*}
        g(x_2)\leq g(\beta)=\frac{b-\theta}{b+D^*}\leq \frac{b}{b+D^*}\quad \forall x_2\geq \beta,
    \end{align*} it follows from \eqref{g'1} that
    \begin{align}
      g'(x_2)\dot x_2 &\leq \frac{D^*b}{b+D^*}(\kappa+1)\max\left\{|g'(x_2)|, \ x_2\geq \beta\right\}e^{x_1}.\label{g'2}
    \end{align}
    And by virtue of \eqref{eqpt-mu1},\eqref{X-S-1} and \eqref{g-p-kappa}, it holds that
    \begin{align*}
        g'(x_2)=\frac{p_0 S_{\text{in}}}{b+D^*}\times \frac{\kappa e^{x_2}}{(\kappa+e^{x_2})^2}\mu'\left(\frac{S_{\text{in}}e^{x_2}}{\kappa+e^{x_2}}\right).
    \end{align*}
    Noticing that for all $x_2\geq \beta $
    \begin{align*}
        \frac{e^{x_2}}{\kappa+e^{x_2}}\leq 1, \quad \text{and}\quad \frac{1}{\kappa+e^{x_2}}\leq \frac{1}{\kappa+e^{\beta}} 
    \end{align*}
    we get for all $x_2\geq \beta,$
    \begin{align}
        |g'(x_2)|\leq \frac{p_0 \kappa S_{\text{in}}}{(b+D^*)(\kappa+e^{\beta})}\left|\mu'\left(\frac{S_{\text{in}}e^{x_2}}{\kappa+e^{x_2}}\right)\right|.\label{g'3}
    \end{align}
    And by setting $\xi=\frac{e^{x_2}}{\kappa+e^{x_2}},$ we obtain for all $x_2\geq \beta$
    \begin{align}
        |g'(x_2)|\leq \frac{p_0 \kappa S_{\text{in}}}{(b+D^*)(\kappa+e^{\beta})}M,\label{g'4}
    \end{align} where
    \begin{align}
        M:=\max\left\{|\mu'\left(S_{\text{in}}\xi\right)|, \ \frac{e^{\beta}}{\kappa+e^{\beta}}\leq \xi\leq 1\right\}.\label{M}
    \end{align}
    Combining \eqref{g'2} with \eqref{g'4}, we obtain for all $x_2\geq \beta$
    \begin{align}
      g'(x_2)\dot x_2  \leq  (\kappa+1)\frac{GbD^*}{b+D^*}e^{x_1}\label{g'x2p}
    \end{align} where 
    \begin{align*}
        G:=1+\frac{p_0\kappa S_{\text{in}}}{(D^*+b)(\kappa+e^{\beta})}M.
    \end{align*}
    Consider the solution of the closed-loop system with initial condition $(x_1(0),x_2(0))\in \mathbb{R}^2$ with 
    \begin{subequations}
        \begin{align}
       x_2(0)=\overline{x}_2,\\
       \theta+(b+D^*)g(\overline{x}_2)+(\kappa+1)\frac{GbD^*}{\theta}e^{x_1(0)}\leq b
    \end{align} \label{initial-cond}
    \end{subequations}
    where $\overline{x}_2$ is the one given in \eqref{assp-theta}.
    
    We next make the following claim.
    \begin{claim}\label{claim-x-2}
        The inequality $x_2(t)\geq \beta$ holds for all $t\geq 0.$
    \end{claim}
    \begin{proof}
    Indeed, if there exists $t>0$ with $x_2(t)<\beta,$ then we may define 
    \begin{align}
        T:=\inf\{t>0:\ x_2(t)<\beta\}.\label{T}
    \end{align}
    Continuity of the solution and definition \eqref{T} implies that 
    \begin{align}
        x_2(T)\leq \beta.\label{x_2-T<}
    \end{align} 
    We have $T\geq 0$ from definition \eqref{T} and since from \eqref{x2beta} and \eqref{initial-cond},  $x_2(0)=\overline{x}_2>\beta,$  we get that $T>0$. Moreover, the minimality of \eqref{T} implies that
    \begin{align}
        x_2(t)\geq \beta, \quad \forall t\in [0,T).\label{x_2T>}
    \end{align} 
    In particular this ensures that $
        \lim_{t\rightarrow T^+} x_2(t)\geq \beta$. 
    By continuity of the solution, it follows that $x_2(T)\geq \beta$.
    Combining this with  \eqref{x_2-T<}, we obtain 
    \begin{align}
        x_2(T)=\beta.\label{x2=beta}
    \end{align} 
    So now in view of \eqref{x_2T>} and \eqref{x2=beta}, $x_2(t)\geq \beta$ for all $t\in [0,T].$ We then obtain from \eqref{x_1p-theta} and by applying comparison Lemma 3.4 in \cite{Khalil2002-Nonlinear}:
    \begin{align}
        x_1(t)\leq x_1(0)-\theta t\label{x_1theta}, \quad \forall t\in [0,T].
    \end{align}
    Inequality \eqref{x_1theta} in conjunction with \eqref{g'x2p} give for all $t\in[0,T]:$
    \begin{align*}
        g'(x_2(t))\dot{x}_2(t)\leq (\kappa+1)\frac{GbD^*}{b+D^*}e^{(x_1(0)-\theta t)},
    \end{align*}
    which is rewritten as
    \begin{align}
        \frac{d}{dt}g(x_2(t))\leq (\kappa+1)\frac{GbD^*}{b+D^*}e^{(x_1(0)-\theta t)}, \quad \forall t\in [0,T].\label{g'x2t}
    \end{align}
    The differential inequality \eqref{g'x2t} directly implies by using comparison Lemma 3.4 in \cite{Khalil2002-Nonlinear} that for all $t\in [0,T]$
    \begin{align*}
        g(x_2(t))\leq g(x_2(0))+(\kappa+1)\frac{GbD^*}{\theta(b+D^*)}e^{x_1(0)}(1-e^{-\theta t}). 
    \end{align*}
Since $(1-e^{-\theta t})<1$ for all $t\in [0,T],$ it holds that
\begin{align}
        g(x_2(t))< g(x_2(0))+(\kappa+1)\frac{GbD^*}{\theta(b+D^*)}e^{x_1(0)}, \quad \forall t\in [0,T]. \label{gx_2t}
    \end{align}
The fact that from \eqref{initial-cond},
    \begin{align*}
        x_2(0)=\overline{x}_2,\quad \theta+(b+D^*)g(\overline{x}_2)+(\kappa+1)\frac{GbD^*}{\theta}e^{x_1(0)}\leq b
    \end{align*} in conjunction with \eqref{gx_2t} give 
    \begin{align*}
        \theta+(b+D^*)g(x_2(t))<b, \quad \forall t\in [0,T].
    \end{align*} 
    This contradicts the fact that from \eqref{bDg}, \eqref{x2=beta}
    \begin{align*}
        (b+D^*)g(\beta)=b-\theta \quad \text{and}\quad  x_2(T)=\beta.
    \end{align*}
    Therefore the inequality $x_2(t)\geq \beta$ holds for all $t\geq 0.$ 
    \end{proof}
    Since the inequality \eqref{x_1p-theta} holds for all $t$ such that $x_2(t)\geq \beta,$ and the Claim \ref{claim-x-2} ensures that $x_2(t)\geq \beta$ for all $t\geq 0,$ the inequality  \eqref{x_1p-theta} holds for all $t\geq 0.$  Using comparison Lemma 3.4 in \cite{Khalil2002-Nonlinear} again, this implies that
        $x_1(t)\leq x_1(0)-\theta t$ for all $t\geq 0$.
    Thus,
    \begin{align*}
        \lim_{t\rightarrow+\infty} x_1(t)=-\infty.
    \end{align*}
    This is a contradiction with the assumption that the solution is bounded and the proof is complete.
    
    \subsection{Proof of Theorem \ref{Th3}}\label{Proof-Th3}
Let $\alpha\in [0,1)$ and $\delta>0$. We define for all $x\in \mathbb{R}^3$ the continuously differentiable function:
    \begin{align}
        V(x)=e^{x_1}-x_1-1+\frac{B}{b+D^*}(e^{x_2}-x_2-1)+Q(x_3), \label{V*}
    \end{align}
where $Q$ is the continuously differentiable function defined by:
    \begin{align}
        Q(x_3):=\begin{cases}
            \frac{1}{2}x_3^2+\frac{\Omega}{2D^*\delta}\int_{x_3}^0\frac{|g(s)-1|^{1-\alpha}}{p(s)g(s)}ds\quad \text{if} \quad x_3\leq 0\\ \\
            R\int_0^{x_3}\frac{e^s}{g(s)^2}(e^s-1)ds, \quad \text{if}\quad x_3>0
        \end{cases}\label{Q*}
    \end{align} and $B, \Omega, R$ are positive constants that will be defined later on. Proceeding in the same way as in the proof of Theorem \ref{Th1}, we can guarantee that the function $Q$ is continuously differentiable, positive definite and radially unbounded such that

    \begin{align}
        Q'(x_3)=\begin{cases}
            x_3-\frac{\Omega}{2D^*\delta}\frac{|g(x_3)-1|^{1-\alpha}}{p(x_3)g(x_3)}, \quad \text{if}\quad x_3\leq 0\\ \\
            \frac{R e^{x_3}}{g(x_3)^2}(e^{x_3}-1), \quad \text{if}\quad x_3>0.
        \end{cases}\label{Q'*}
    \end{align}
    
    Since the function $x_1\mapsto e^{x_1}-x_1-1,$ is positive definite and radially unbounded as well as the function $x_2\mapsto \frac{B}{b+D^*}(e^{x_2}-x_2-1)$ and the function $Q,$ it follows that the function  $V$ defined by \eqref{V*} is positive definite and radially unbounded.


     Recall that we have from \eqref{systeme-age-stud}, \eqref{feedback2} and \eqref{u*} the following equations:
    \begin{align}
        D&=(\kappa+1)D^*g(x_3)\frac{e^{x_1-x_3}}{p(x_3)}+u(x_2)\label{D-age}\\
        \dot{x}_3&=D^*g(x_3)e^{x_1-x_3}(1-e^{x_3})+p(x_3)u(x_2).\label{x_3-age-close}
    \end{align} where
    \begin{align}
        u(x_2):=\begin{cases}
            \delta|g(x_3)-1|^{1+\alpha}, \quad \text{if}\quad x_3\leq 0\\
            0, \quad \text{if}\quad x_3>0
        \end{cases}\label{u*}
    \end{align}
    

    From, \eqref{systeme-age-stud}, \eqref{x_3-age-close} and \eqref{V*} we have for all $x\in \mathbb{R}^3$
    \begin{eqnarray}
        \dot{V}(x)&=&(b+D^*)g(x_3)(e^{x_1}-1)(e^{x_2}-1)+b(g(x_3)-1)(e^{x_1}-1)+B\lambda(1-g(x_3))(e^{x_2}-1)\nonumber\\
        &&-B(\lambda e^{-x_2}+g(x_3))(e^{x_2}-1)^2-D^*g(x_3)(e^{x_1}-1)^2-D^*g(x_3)e^{x_1-x_3}Q'(x_3)(e^{x_3}-1)\nonumber\\
        &&+Q'(x_3)p(x_3)u(x_2).\label{Vp1*}
    \end{eqnarray}
    Using the Young inequalities
    \begin{align*}
        (e^{x_1}-1)(e^{x_2}-1)\leq \frac{D^*}{4(b+D^*)}(e^{x_1}-1)^2
        +\frac{(b+D^*)}{D^*}(e^{x_2}-1)^2\\
        b(g(x_3)-1)(e^{x_1}-1)\leq \frac{D^*g(x_3)}{4}(e^{x_1}-1)^2
        +\frac{b^2}{D^*g(x_3)}(g(x_3)-1)^2\\
        (e^{x_2}-1)(1-g(x_3))\leq \frac{(b+D^*)^2g(x_3)}{BD^*}(e^{x_2}-1)^2
        +\frac{BD^*(g(x_3)-1)^2}{4(b+D^*)^2g(x_3)},
    \end{align*}
    we get 
    \begin{eqnarray*}
        \dot{V}(x)
        &\leq \frac{1}{g(x_3)}\left[\frac{b^2}{D^*}+\frac{B^2D^*\lambda}{4(b+D^*)^2}\right](g(x_3)-1)^2-B\left[\lambda e^{-x_2}+g(x_3)-(1+\lambda)\frac{(b+D^*)^2}{D^*B}g(x_3)\right](e^{x_2}-1)^2\\
        &-\frac{D^*}{2}g(x_3)(e^{x_1}-1)^2
        -D^* g(x_3)e^{x_1-x_3}Q'(x_3)(e^{x_3}-1)+Q'(x_3)p(x_3)u(x_2).
    \end{eqnarray*}
    And since $-B\lambda e^{-x_2}(e^{x_2}-1)^2<0,$ we obtain
    \begin{eqnarray*}
        \dot{V}(x)
        &\leq& \frac{1}{g(x_3)}\left[\frac{b^2}{D^*}+\frac{B^2D^*\lambda}{4(b+D^*)^2}\right](g(x_3)-1)^2-Bg(x_3)\left[1-(1+\lambda)\frac{(b+D^*)^2}{D^*B}\right](e^{x_2}-1)^2\\
        &&-\frac{D^*}{2}g(x_3)(e^{x_1}-1)^2
        -D^* g(x_3)e^{x_1-x_3}Q'(x_3)(e^{x_3}-1)+Q'(x_3)p(x_3)u(x_2).
    \end{eqnarray*}
    By picking
    \begin{align}
        \Omega=2b^2+\frac{\lambda B^2{D^*}^2}{2(b+D^*)^2}, \quad B=(1+\lambda)\frac{\varphi(b+D^*)^2}{D^*},\label{Omega-B}
    \end{align}
    it follows that
    \begin{eqnarray}
        \dot{V}(x)
        &\leq& \Omega\frac{(g(x_3)-1)^2}{2D^* g(x_3)}-(\varphi-1)(1+\lambda)\frac{(b+D^*)^2 }{D^*}g(x_3)(e^{x_2}-1)^2-\frac{D^*}{2}g(x_3)(e^{x_1}-1)^2\nonumber\\
        &&-D^* g(x_3)e^{x_1-x_3}Q'(x_3)(e^{x_3}-1)+Q'(x_3)p(x_3)u(x_2).\label{Vp2*}
    \end{eqnarray}

The following observation is established in Section \ref{proof-Claim-Th3}.
    
    \begin{claim}\label{Claim-Th3}
        If assumption \textbf{(C)} holds then the quantity
    \begin{align}
        r:=D^*\left(1-\frac{1}{4\varphi(1+\lambda)}\right)-\sup_{\xi\geq 0}\left(\left(b-\frac{\lambda(b+D^*)}{2}\right)\frac{1-g(\xi)}{g(\xi)}+(1+\lambda)\frac{\lambda^2\varphi(b+D^*)^2}{4D^*}\left(\frac{1-g(\xi)}{g(\xi)}\right)^2\right)\label{r}
    \end{align}
    is positive.
    \end{claim}
    Picking any constant $c>0$ for which
    \begin{align}
        \left(b+\frac{\lambda(b+D^*)}{2}\right)\max\left\{\frac{|g(x_3)-1|}{g(x_3)}: \ x_3\geq 0\right\}+2D^*<re^{c},\label{c*}
    \end{align} 
    we distinguish the following cases.
    \begin{enumerate}
        \item[]Case $1:$ $x_3\leq 0.$ In view of \eqref{x_1-x_2-x_3}, this case corresponds to $S\leq S^*$ which means the biomass is underfed. In this case we obtain from \eqref{u*} and \eqref{Vp2*}:
        \begin{eqnarray*}
        \dot{V}(x)
        &\leq& \Omega\frac{(g(x_3)-1)^2}{2D^* g(x_3)}(\varphi-1)(1+\lambda)\frac{(b+D^*)^2 }{D^*}g(x_3)(e^{x_2}-1)^2-\frac{D^*}{2}g(x_3)(e^{x_1}-1)^2\\
        &&-D^* g(x_3)e^{x_1-x_3}Q'(x_3)(e^{x_3}-1)+\delta Q'(x_3)p(x_3)|g(x_3)-1|^{1+\alpha}.
    \end{eqnarray*}
    Using \eqref{Q'*}, it holds that
    \begin{eqnarray*}
        \dot{V}(x)
        &\leq& -(\varphi-1)(1+\lambda)\frac{(b+D^*)^2 }{D^*}g(x_3)(e^{x_2}-1)^2-\frac{D^*}{2}g(x_3)(e^{x_1}-1)^2\\
        &&-D^* g(x_3)e^{x_1-x_3}x_3(e^{x_3}-1)-\frac{\Omega e^{x_1-x_3}}{2\delta p(x_3)}|g(x_3)-1|^{1-\alpha}(1-e^{x_3})\\
        &&+\delta x_3 p(x_3)|g(x_3)-1|^{1+\alpha}-\frac{\Omega}{2D^*}\frac{(g(x_3)-1)^2}{g(x_3)}.
    \end{eqnarray*}
Since
    \begin{align*}
        x_3\leq 0,\quad p(x_3)>0, \quad \text{and}\quad -\frac{\Omega}{2D^*}\frac{(g(x_3)-1)^2}{g(x_3)}\leq 0,
    \end{align*} 
    it follows that
        \begin{eqnarray}
            \dot{V}(x)&\leq& -(\varphi-1)(1+\lambda)\frac{(b+D^*)^2}{D^*}g(x_3)(e^{x_2}-1)^2-\frac{D^*}{2}g(x_3)(e^{x_1}-1)^2\nonumber\\
            &&-\left[D^*g(x_3)|x_3|+\frac{\Omega}{2\delta p(x_3)}|g(x_3)-1|^{1-\alpha}\right]e^{x_1}(e^{-x_3}-1).\label{Case1*}
        \end{eqnarray}
        \item[]Case $2:$ $x_3>0$ and $x_1\geq -c.$ 
This case corresponds to the current substrate being both {\em in excess of the equilibrium level}, $S>S^*$, and {\em providing sufficient nurishnment even if the current biomass $X$ exhibits underpopulation}, $S\geq  S_{\text{in}}-\frac{\mu(S^*)e^c}{D^*}X$. 
        In this case we obtain from \eqref{Q'*}, \eqref{u*},  and \eqref{Vp2*}:
            \begin{eqnarray*}
            \dot{V}(x)&\leq& \frac{\Omega }{2D^*g(x_3)}(g(x_3)-1)^2-(\varphi-1)(1+\lambda)\frac{(b+D^*)^2}{D^*}g(x_3)(e^{x_2}-1)^2\\
            &&-\frac{D^*}{2}g(x_3)(e^{x_1}-1)^2-e^{x_1}\frac{RD^*}{g(x_3)}(e^{x_3}-1)^2.
        \end{eqnarray*}
        So using \eqref{Glipscitz} and the fact that $x_1\geq -c,$ we get
        \begin{eqnarray*}
            \dot{V}(x) &\leq &\frac{\Omega A^2}{2D^*g(x_3)}(e^{x_3}-1)^2-(\varphi-1)(1+\lambda)\frac{(b+D^*)^2}{D^*}g(x_3)(e^{x_2}-1)^2\\
            &&-\frac{D^*}{2}g(x_3)(e^{x_1}-1)^2-e^{-c}\frac{RD^*}{g(x_3)}(e^{x_3}-1)^2.
        \end{eqnarray*}
        Consequently, selecting $R$ large enough that $R>\frac{A^2\Omega e^{c}}{{D^*}^2},$ we get
        \begin{eqnarray}
            \dot{V}(x)&\leq& -(\varphi-1)(1+\lambda)\frac{(b+D^*)^2}{D^*}g(x_3)(e^{x_2}-1)^2-\frac{D^*}{2}g(x_3)(e^{x_1}-1)^2\nonumber\\
            &&-e^{-c}\frac{RD^*}{2g(x_3)}(e^{x_3}-1)^2. \label{Case2*}
        \end{eqnarray}
        \item[]Case $3:$ $x_3>0$ and $x_1<-c.$ 
In the original coordinates, this case corresponds to the current substrate being {\em in excess of the equilibrium level}, $S>S^*$, but nevertheless {\em undernurishing the currently underpopulated biomass $X$}, $S<  S_{\text{in}}-\frac{\mu(S^*)e^c}{D^*}X$. 
In this case we obtain from \eqref{Q'*}, \eqref{u*}, \eqref{Vp1*}:
        \begin{eqnarray}
            \dot{V}(x)&=&(b+D^*)g(x_3)(e^{x_1}-1)(e^{x_2}-1)+b(g(x_3)-1)(e^{x_1}-1)+B\lambda(1-g(x_3))(e^{x_2}-1)\nonumber\\
            &&-B(\lambda e^{-x_2}+g(x_3))(e^{x_2}-1)^2
            -D^*g(x_3)(e^{x_1}-1)^2-\frac{RD^*e^{x_1}}{g(x_3)}(e^{x_3}-1)^2.\label{Vp3*}
        \end{eqnarray}
        Consider the following fact that holds for all $y\in \mathbb{R}$
        \begin{align}
            &(b+D^*)g(x_3)(e^{x_1}-1)y+B\lambda(1-g(x_3))y-Bg(x_3)y^2\leq \frac{(b+D^*)^2}{4B}g(x_3)(e^{x_1}-1)^2+\nonumber\\
            &\frac{\lambda(b+D^*)}{2}(1-g(x_3))(e^{x_1}-1)+\frac{B\lambda^2}{4g(x_3)}(1-g(x_3))^2.\label{Fact}
        \end{align} This is the consequence of the fact that the function $y\in \mathbb{R}\mapsto \omega_1 y-\omega_2 y^2, \ (\omega_2>0)$ reaches its maximum at $y_{\max}=\frac{\omega_1}{2\omega_2}$ where
        \begin{align*}
            \omega_1&:=\left[(b+D^*)g(x_3)(e^{x_1}-1)+B\lambda(1-g(x_3))\right], \\ \omega_2&:=Bg(x_3).
        \end{align*}
        Then, we obtain from \eqref{Fact}, by setting $y=(e^{x_2}-1)$
        \begin{eqnarray*}
            (b+D^*)g(x_3)(e^{x_1}-1)(e^{x_2}-1)+B\lambda(1-g(x_3))(e^{x_2}-1)
            -Bg(x_3)(e^{x_2}-1)^2&\leq&\\
            \frac{(b+D^*)^2}{4B}g(x_3)(e^{x_1}-1)^2+
            \frac{\lambda(b+D^*)}{2}(1-g(x_3))(e^{x_1}-1)+\frac{B\lambda^2}{4g(x_3)}(1-g(x_3))^2
        \end{eqnarray*} and we get from \eqref{Vp3*}:
        \begin{eqnarray*}
            \dot{V}(x)&\leq& \left(b-\frac{\lambda(b+D^*)}{2}\right)(g(x_3)-1)(e^{x_1}-1)
            -\left(D^*-\frac{(b+D^*)^2}{4B}\right)g(x_3)(e^{x_1}-1)^2\\
            &&-B\lambda e^{-x_2}(e^{x_2}-1)^2
            -\frac{RD^*e^{x_1}}{g(x_3)}(e^{x_3}-1)^2+\frac{B\lambda^2}{4g(x_3)}(1-g(x_3))^2.
        \end{eqnarray*}
        Since for all $x_3>0, x_2\in \mathbb{R}, x_1<-c$
        \begin{align*}
            -B\lambda e^{-x_2}(e^{x_2}-1)^2
            -\frac{RD^*e^{x_1}}{g(x_3)}(e^{x_3}-1)^2<0,
        \end{align*} we get
        \begin{eqnarray*}
            \dot{V}(x)&\leq& \left(b-\frac{\lambda(b+D^*)}{2}\right)(g(x_3)-1)(e^{x_1}-1)
            -\left(D^*-\frac{(b+D^*)^2}{4B}\right)g(x_3)(e^{x_1}-1)^2\\
            &&+\frac{B\lambda^2}{4g(x_3)}(1-g(x_3))^2. 
        \end{eqnarray*}
        So, using \eqref{Omega-B}, we obtain
        \begin{eqnarray}
            \frac{\dot{V}(x)}{g(x_3)}&\leq&\left(b-\frac{\lambda(b+D^*)}{2}\right)\frac{g(x_3)-1}{g(x_3)}(e^{x_1}-1)-D^*\left(1-\frac{1}{4(1+\lambda)\varphi}\right)(e^{x_1}-1)^2\nonumber\\
            &&+(1+\lambda)\frac{\lambda^2\varphi(b+D^*)^2}{4D^*}\left(\frac{g(x_3)-1}{g(x_3)}\right)^2\nonumber\\
            &=& \left(b-\frac{\lambda(b+D^*)}{2}\right)\frac{1-g(x_3)}{g(x_3)}+(1+\lambda)\frac{\lambda^2\varphi(b+D^*)^2}{4D^*}\left(\frac{1-g(x_3)}{g(x_3)}\right)^2\nonumber\\
            &&-D^*\left(1-\frac{1}{4(1+\lambda)\varphi}\right)+e^{x_1}\left(\left(b-\frac{\lambda(b+D^*)}{2}\right)\frac{g(x_3)-1}{g(x_3)}\right.\nonumber\\
            &&\left.+2D^*\left(1-\frac{1}{4(1+\lambda)\varphi}\right)-D^*e^{x_1}\left(1-\frac{1}{4(1+\lambda)\varphi}\right)\right).\label{Vp4x}
        \end{eqnarray}
    Let us set
        \begin{align}
            P(x_3):=\left(b-\frac{\lambda(b+D^*)}{2}\right)\frac{1-g(x_3)}{g(x_3)}
            +(1+\lambda)\frac{\lambda^2\varphi(b+D^*)^2}{4D^*}\left(\frac{1-g(x_3)}{g(x_3)}\right)^2.\label{P(X_3)}
        \end{align}
        Then we rewrite \eqref{Vp4x} as follows
        \begin{eqnarray}
           \frac{\dot{V}(x)}{g(x_3)}&\leq& P(x_3) -D^*\left(1-\frac{1}{4(1+\lambda)\varphi}\right)+e^{x_1}\left(\left(b-\frac{\lambda(b+D^*)}{2}\right)\frac{g(x_3)-1}{g(x_3)}\right.\nonumber\\
           &&\left.+2D^*\left(1-\frac{1}{4(1+\lambda)\varphi}\right)-D^*e^{x_1}\left(1-\frac{1}{4(1+\lambda)\varphi}\right)\right).\label{Vp4*}
        \end{eqnarray}
        So thanks to the fact that
        \begin{align*}
            1-\frac{1}{4(1+\lambda)\varphi}<1, \quad -e^{x_1}D^*\left(1-\frac{1}{4(1+\lambda)\varphi}\right)<0,
        \end{align*} we get from \eqref{Vp4*} that
        \begin{align}
          &\frac{\dot{V}(x)}{g(x_3)}\leq P(x_3)-D^*\left(1-\frac{1}{4(1+\lambda)\varphi}\right)+e^{x_1}\left(\left(b-\frac{\lambda(b+D^*)}{2}\right)\frac{g(x_3)-1}{g(x_3)}+2D^*\right).\label{Vp4**}  
        \end{align}
        Definition \eqref{P(X_3)} in conjunction with \eqref{r} implies that 
        \begin{align*}
            P(x_3)-D^*\left(1-\frac{1}{4(1+\lambda)\varphi}\right)\leq -r.
        \end{align*}
        Consequently, we get from \eqref{Vp4**} and the fact that $x_1\leq -c:$
        \begin{align}
            &\dot{V}(x)\leq -rg(x_3)+e^{-c}g(x_3)\left(\left(b+\frac{\lambda(b+D^*)}{2}\right)\frac{|g(x_3)-1|}{g(x_3)}+2D^*\right).\label{Vp5*}
        \end{align}
        Hence  \eqref{c*} in conjunction with \eqref{Vp5*}, show that 
        \begin{align}
            \dot{V}(x)<0.\label{Case3*}
        \end{align}
    \end{enumerate}
    Combining all cases, we conclude from \eqref{Case1*}, \eqref{Case2*} and \eqref{Case3*} that $\dot{V}(x)<0$ for all $x\in \mathbb{R}^3\setminus\{0\}.$ Therefore, by Lyapunov's theorem $0\in \mathbb{R}^3$ is globally asymptotically stable  for the closed-loop system \eqref{systeme-age-stud} with \eqref{feedback2}.

    If $\alpha\in (0,1),$ the function $u$ defined in \eqref{u*} is continuously differentiable as stated in the proof of Theorem \ref{Th1}. 
    Using \eqref{g}, \eqref{p}, \eqref{u*}, the Jacobian matrix of the closed-loop system \eqref{syst-age-x1}-\eqref{syst-age-x2}-\eqref{x_3-age-close} is given by 
    \begin{align*}
        \begin{pmatrix}
            &-D^* \quad &b+D^*\quad &bg'(0)\\
            &0\quad &-2\lambda \quad &-\lambda (b+D^*)g'(0)\\
            &0 \quad &0\quad &-D^*
        \end{pmatrix}.
    \end{align*}
The eigenvalues of this matrix are $-D^*<0$ and $-2\lambda<0$ (recall \eqref{lambda}). So it is Hurwitz and the linearization of the closed-loop system \eqref{syst-age-x1}-\eqref{syst-age-x2}-\eqref{x_3-age-close} is exponentially stable at the origin $0\in \mathbb{R}^3$. Therefore, by using the same arguments as in the proof of Theorem \ref{Th1}, we can conclude that the feedback law \eqref{feedback2} achieves local exponential stabilization of the system \eqref{systeme-age-stud} at $0\in \mathbb{R}^3.$

    \section{Conclusions}\label{Conclusions}
    The global asymptotic stabilization problem with positive-valued feedback is addressed and solved in this paper for chemostat models with non-zero mortality rate. We have proposed feedback laws inspired from the one proposed in \cite{Mailleret2001-A-simple}, \cite{Karafyllis2008-Vector} to achieve global asymptotic stabilization of lumped and age structured-chemostat models. The proposed feedback laws work under conditions that guarantee that the mortality rate is sufficiently small. 
    
    In the lumped case, we have shown that the condition that guarantees existence of a global stabilizer is close to being a necessary condition for a wide class of kinetic equations for the specific growth rate including the case of Haldane kinetics. 
    
    The age-structured model studied in this paper is shown to be equivalent to a three-state chemostat model. The general infinite-dimensional case needs to be studied further. However, we strongly believe that the results provided in the present work can be used as a basis for the study of the general case.

    \textbf{Acknowledgments:} The authors would like to thank Dr Dionysios Theodosis for his help with the phase diagrams shown in Section \ref{Sec-Simulation}.

    \section{Appendix}
    \subsection{Proof of Claim \ref{Claim-Lipschitz}}
     The fact that $\mu$ is $C^1$ on $[0,S_{\text{in}}]$ implies the existence of $L\geq 0$ such that
     \begin{align}
         |\mu(S)-\mu(S^*)|\leq L|S-S^*|, \quad \forall S\in [0,S_{\text{in}}].\label{muL}
     \end{align}
     So in view of \eqref{X-S-1}-\eqref{g-p-kappa}, the inequality \eqref{muL} yields for all $z\in \mathbb{R}$
     \begin{align*}
         |g(z)-1|=\left|\frac{\mu(S)}{\mu(S^*)}-1\right|&\leq \frac{L}{\mu(S^*)}|S-S^*|\\
         &=\frac{Lp_0S^*}{b+D^*}\left|\frac{S_{\text{in}}e^{z}}{S_{\text{in}}-S^*+S^*e^{z}}-1\right|\\
         &=\frac{Lp_0S^*}{b+D^*}\times \frac{1}{1+\frac{1}{\kappa}e^{z}}|e^{z}-1|.
     \end{align*}
     Since $1/\left(1+\frac{1}{\kappa}e^{z}\right)<1$ we get for all $z\in \mathbb{R},$
     \begin{align*}
      |g(z)-1|&\leq   \frac{Lp_0S^*}{b+D^*} |e^{z}-1|\\
      &\leq A|e^{z}-1|,
     \end{align*}
     where \begin{align*}
         A:=\frac{Lp_0S^*}{b+D^*}.
     \end{align*}  

\subsection{Proof of the Claim \ref{Claim-Th1}}

  If assumption \textbf{(A)} is valid, then by continuity of $\mu:\mathbb{R}_+\rightarrow\mathbb{R}_+$ it holds that
 \begin{align*}
     \min_{S\in [S^*, S_{\text{in}}]}\ (p_0\mu(S))>b.
 \end{align*}
 Since from \eqref{eqpt-mu1}, $p_0\mu(S^*)=b+D^*,$ we get that
 \begin{align*}
     (D^*+b)\min_{S\in [S^*, S_{\text{in}}]}\ \mu(S)>b\mu(S^*).
 \end{align*} 
 Consequently, there exists $r\in [0,1)$ such that 
 \begin{align*}
     (rD^*+b)\min_{S\in [S^*, S_{\text{in}}]}\ \mu(S)\geq b\mu(S^*).
 \end{align*}
 Therefore, 
 \begin{align*}
     (rD^*+b) \mu(S)\geq b\mu(S^*) \quad \forall S\in [S^*, S_{\text{in}}].
 \end{align*}
 Noticing that from \eqref{X-S-1} and \eqref{g-p-kappa}, $g(x_2)=\frac{\mu(S)}{\mu(S^*)},$ and $S\in [S^*,S_{\text{in}})$ is equivalent to $x_2\geq 0,$  we get that 
 \begin{align*}
     g(x_2)\geq \frac{b}{b+rD^*} \quad \forall x_2\geq 0,
 \end{align*} which implies
 \begin{align*}
     b\frac{g(x_2)-1}{g(x_2)}\geq -rD^*, \quad \forall x_2\geq 0.
 \end{align*} 
\subsection{Proof of the Claim \ref{Claim-Th3}}\label{proof-Claim-Th3}

 Definition \eqref{g-p-kappa} and \eqref{X-Y-S} show that  \eqref{Assump-C} is equivalent to the following inequality:
     \begin{align}
         \sup_{\xi\geq 0}\left(\left(b-\frac{\lambda(b+D^*)}{2}\right)\frac{1-g(\xi)}{g(\xi)}+(1+\lambda)\frac{\lambda^2\varphi(b+D^*)^2}{4D^*}\left(\frac{1-g(\xi)}{g(\xi)}\right)^2\right)<D^*\left(1-\frac{1}{4\varphi(1+\lambda)}\right).\label{AssumpC-equiv}
     \end{align}
     In view of the definition \eqref{r} of $r,$
     we notice that \eqref{AssumpC-equiv} implies that $r$ is positive. 

     
\bibliography{References.bib}

\begin{thebibliography}{10}

\bibitem{Amster2020-dynamics}
P.~Amster, G.~Robledo, and D.~Sep{\'u}lveda.
\newblock Dynamics of a chemostat with periodic nutrient supply and delay in
  the growth.
\newblock {\em Nonlinearity}, 33:5839, 2020.

\bibitem{Bastin2013-Line}
G.~Bastin and D.~Dochain.
\newblock {\em On-line estimation and adaptive control of bioreactors},
  volume~1.
\newblock Elsevier, 2013.

\bibitem{Beauthier2015-Input}
C.~Beauthier, J.~J. Winkin, and D.~Dochain.
\newblock Input/state invariant {LQ}-optimal control: Application to
  competitive coexistence in a chemostat.
\newblock {\em Evolution Equations \& Control Theory}, 4, 2015.

\bibitem{Borisov2022-stability}
M.~Borisov, N.~Dimitrova, and P.~Zlateva.
\newblock Stability analysis of a chemostat model for phenol and sodium
  salicylate mixture biodegradation.
\newblock {\em Processes}, 10:2571, 2022.

\bibitem{Borisov2020-Global}
M.~K. Borisov, N.~S. Dimitrova, and M.~I. Krastanov.
\newblock Global stabilizability of an anaerobic biodegradation process via
  piecewise constant feedback.
\newblock {\em International Journal of Robust and Nonlinear Control},
  30:2777--2795, 2020.

\bibitem{Boucekkine2011-Optimal}
R.~Boucekkine, N.~Hritonenko, and Y.~Yatsenko.
\newblock {\em Optimal control of age-structured populations in economy,
  demography, and the environment}.
\newblock Routledge, 2011.

\bibitem{Brauer2012-Mathematical}
F.~Brauer and C.~Castillo-Chavez.
\newblock {\em Mathematical models in population biology and epidemiology},
  volume~2.
\newblock Springer, 2012.

\bibitem{Butler1985-Mathematical}
G.~J. Butler and G.~S.~K. Wolkowicz.
\newblock A mathematical model of the chemostat with a general class of
  functions describing nutrient uptake.
\newblock {\em SIAM Journal on Applied Mathematics}, 45:138--151, 1985.

\bibitem{De2003-Feedback}
P.~De~Leenheer and H.~Smith.
\newblock Feedback control for chemostat models.
\newblock {\em Journal of Mathematical Biology}, 46:48--70, 2003.

\bibitem{Dimitrova2012-Nonlinear}
N.~Dimitrova and M.~Krastanov.
\newblock Nonlinear adaptive stabilizing control of an anaerobic digestion
  model with unknown kinetics.
\newblock {\em International Journal of Robust and Nonlinear Control},
  22:1743--1752, 2012.

\bibitem{Dochain2013-Automatic}
D.~Dochain.
\newblock {\em Automatic control of bioprocesses}.
\newblock John Wiley \& Sons, 2013.

\bibitem{Feichtinger2003-Optimality}
G.~Feichtinger, G.~Tragler, and V.~M. Veliov.
\newblock Optimality conditions for age-structured control systems.
\newblock {\em Journal of Mathematical Analysis and Applications}, 288:47--68,
  2003.

\bibitem{Gouze2006-feedback}
J.-L. Gouze and G.~Robledo.
\newblock Feedback control for competition models with mortality in the
  chemostat.
\newblock In {\em Proceedings of the 45th IEEE Conference on Decision and
  Control}, pages 2098--2103. IEEE, 2006.

\bibitem{Haacker2024-stabilization}
P.-E. Haacker, I.~Karafyllis, M.~Krsti{\'c}, and M.~Diagne.
\newblock Stabilization of age-structured chemostat hyperbolic {PDE} with
  actuator dynamics.
\newblock {\em International Journal of Robust and Nonlinear Control},
  34:6741--6763, 2024.

\bibitem{Harmand2006-Dynamical}
J.~Harmand, D.~Dochain, and M.~Guay.
\newblock Dynamical optimization of a configuration of multi-fed interconnected
  bioreactors by optimum seeking.
\newblock In {\em Proceedings of the 45th IEEE Conference on Decision and
  Control}, pages 2122--2127. IEEE, 2006.

\bibitem{Harmand2017-chemostat}
J.~Harmand, C.~Lobry, A.~Rapaport, and T.~Sari.
\newblock {\em The chemostat: Mathematical theory of microorganism cultures},
  volume~1.
\newblock John Wiley \& Sons, 2017.

\bibitem{Harmand2020-Increasing}
J.~Harmand, A.~Rapaport, and D.~Dochain.
\newblock Increasing the dilution rate can globally stabilize two-step
  biological systems.
\newblock {\em Journal of Process Control}, 95:67--74, 2020.

\bibitem{Karafyllis2012-New}
I.~Karafyllis and Z.-P. Jiang.
\newblock A new small-gain theorem with an application to the stabilization of
  the chemostat.
\newblock {\em International Journal of Robust and Nonlinear Control},
  22:1602--1630, 2012.

\bibitem{Karafyllis2009-Relaxed}
I.~Karafyllis, C.~Kravaris, and N.~Kalogerakis.
\newblock Relaxed {L}yapunov criteria for robust global stabilisation of
  non-linear systems.
\newblock {\em International Journal of Control}, 82:2077--2094, 2009.

\bibitem{Karafyllis2008-Vector}
I.~Karafyllis, C.~Kravaris, L.~Syrou, and G.~Lyberatos.
\newblock A vector {L}yapunov function characterization of input-to-state
  stability with application to robust global stabilization of the chemostat.
\newblock {\em European Journal of Control}, 14:47--61, 2008.

\bibitem{Karafyllis2017-Stability}
I.~Karafyllis and M.~Krsti{\'c}.
\newblock Stability of integral delay equations and stabilization of
  age-structured models.
\newblock {\em ESAIM: Control, Optimisation and Calculus of Variations},
  23:1667--1714, 2017.

\bibitem{Khalil2002-Nonlinear}
H.~K. Khalil.
\newblock {\em Nonlinear systems}.
\newblock Prentice Hall, 2002.

\bibitem{Mailleret2001-A-simple}
L.~Mailleret and O.~Bernard.
\newblock A simple robust controller to stabilise an anaerobic digestion
  process.
\newblock {\em IFAC Proceedings Volumes}, 34:207--212, 2001.
\newblock 8th IFAC International Conference on Computer Applications in
  Biotechnology 2001, Québec, Canada, 24-27 June 2001.

\bibitem{Mazenc2010-Stabilization}
F.~Mazenc and M.~Malisoff.
\newblock Stabilization of a chemostat model with {H}aldane growth functions
  and a delay in the measurements.
\newblock {\em Automatica}, 46:1428--1436, 2010.

\bibitem{Mazenc2010-Stabilization-2-species}
F.~Mazenc and M.~Malisoff.
\newblock Stabilization of two-species chemostats with delayed measurements and
  {H}aldane growth functions.
\newblock In {\em Proceedings of the 2010 American Control Conference}, pages
  6740--6744. IEEE, 2010.

\bibitem{Mazenc2012-stability}
F.~Mazenc and M.~Malisoff.
\newblock Stability and stabilization for models of chemostats with multiple
  limiting substrates.
\newblock {\em Journal of Biological Dynamics}, 6:612--627, 2012.

\bibitem{Mazenc2024-stability}
F.~Mazenc, G.~Robledo, and D.~Sepulveda.
\newblock A stability analysis of a time-varying chemostat with pointwise
  delay.
\newblock {\em Mathematical Biosciences and Engineering}, 21:38, 2024.

\bibitem{RAHAHA}
A.~Rapaport, H.~Haidar, and J.~Harmand.
\newblock Global dynamics of the buffered chemostat for a general class of
  response functions.
\newblock {\em Journal of Mathematical Biology}, 71:69--98, 2015.

\bibitem{Robledo2012-global}
G.~Robledo, F.~Grognard, and J.-L. Gouz{\'e}.
\newblock Global stability for a model of competition in the chemostat with
  microbial inputs.
\newblock {\em Nonlinear Analysis: Real World Applications}, 13:582--598, 2012.

\bibitem{Rundnicki1994-Asymptotic}
R.~Rundnicki and M.~C. Mackey.
\newblock Asymptotic similarity and malthusian growth in autonomous and
  nonautonomous populations.
\newblock {\em Journal of Mathematical Analysis and Applications},
  187:548--566, 1994.

\bibitem{Schmidt2018-Yield}
K.~Schmidt, I.~Karafyllis, and M.~Krsti{\'c}.
\newblock Yield trajectory tracking for hyperbolic age-structured population
  systems.
\newblock {\em Automatica}, 90:138--146, 2018.

\bibitem{Smith1995-Theory}
H.~L. Smith and P.~Waltman.
\newblock {\em The theory of the chemostat: dynamics of microbial competition},
  volume~13.
\newblock Cambridge University Press, 1995.

\bibitem{Sun2014-Optimal}
B.~Sun.
\newblock Optimal control of age-structured population dynamics for spread of
  universally fatal diseases ii.
\newblock {\em Applicable Analysis}, 93:1730--1744, 2014.

\bibitem{Toth2006-Limit}
D.~Toth and M.~Kot.
\newblock Limit cycles in a chemostat model for a single species with age
  structure.
\newblock {\em Mathematical Biosciences}, 202:194--217, 2006.

\bibitem{Veil2024-Stabilization}
C.~Veil, M.~Krsti{\'c}, I.~Karafyllis, M.~Diagne, and O.~Sawodny.
\newblock Stabilization of predator-prey age-structured hyperbolic {PDE} when
  harvesting both species is inevitable.
\newblock {\em arXiv preprint arXiv:2410.06823}, 2024.

\bibitem{Veldkamp1977-ecological}
H.~Veldkamp.
\newblock Ecological studies with the chemostat.
\newblock In {\em Advances in microbial ecology}, pages 59--94. Springer, 1977.

\bibitem{Wang2006-Delayed}
L.~Wang and G.~S.~K. Wolkowicz.
\newblock A delayed chemostat model with general nonmonotone response functions
  and differential removal rates.
\newblock {\em Journal of Mathematical Analysis and Applications},
  321:452--468, 2006.

\bibitem{Wolkowicz1992-Global}
G.~S.~K. Wolkowicz and Z.~Lu.
\newblock Global dynamics of a mathematical model of competition in the
  chemostat: general response functions and differential death rates.
\newblock {\em SIAM Journal on Applied Mathematics}, 52:222--233, 1992.

\bibitem{Ye2022-periodic}
N.~Ye, Z.~Hu, and Z.~Teng.
\newblock Periodic solution and extinction in a periodic chemostat model with
  delay in microorganism growth.
\newblock {\em Communications on Pure and Applied Analysis}, 21:1361--1384,
  2022.

\bibitem{Ye2023-dynamical}
N.~Ye, L.~Zhang, and Z.~Teng.
\newblock The dynamical behavior and periodic solution in delayed nonautonomous
  chemostat models.
\newblock {\em Journal of Applied Analysis and Computation}, 13:156--183, 2023.

\end{thebibliography}
\bibliographystyle{plain}

\end{document}